\documentclass[pdflatex,sn-mathphys-num,12pt]{sn-jnl}

\usepackage{graphicx}%
\usepackage{multirow}%
\usepackage{amsmath,amssymb,amsfonts}%
\usepackage{bbm}%
\usepackage{amsthm}%
\usepackage{mathrsfs}%
\usepackage[title]{appendix}%
\usepackage{xcolor}%
\usepackage{textcomp}%
\usepackage{manyfoot}%
\usepackage{booktabs}%
\usepackage{algorithm}%
\usepackage{algpseudocode}%
\usepackage{listings}%
\usepackage{wrapfig}%
\usepackage{setspace}%

\theoremstyle{thmstyleone}%
\newtheorem{theorem}{Theorem}
\theoremstyle{thmstyletwo}%

\theoremstyle{thmstylethree}%
\newtheorem{definition}{Definition}%

\newcommand{\half}{\mbox{\small $\frac{1}{2}$}}
\newcommand{\sect}{\S}

\begin{document}

\title[Rational GP approximation]{Rational approximation and 
intrinsic\\ Gaussian processes}


\author*[1]{\fnm{Christopher} \sur{Beattie}}\email{beattie@vt.edu}

\author[2]{\fnm{David} \sur{Higdon}}\email{dhigdon@vt.edu}

\author[2]{\fnm{Leanna} \sur{House}}\email{lhouse@vt.edu}
\author[2]{\fnm{Colby} \sur{Stakun-Pickering}}\email{colbysp32@vt.edu}
\author[2]{\fnm{Jared} \sur{Clark}}\email{cjared96@vt.edu}

\affil*[1]{\orgdiv{Department of Mathematics}, \orgname{Virginia Tech}, \orgaddress{\street{225 Stanger Street}, \city{Blacksburg}, \postcode{24061}, \state{VA}, \country{USA}}}

\affil[2]{\orgdiv{Department of Statistics}, \orgname{Virginia Tech}, \orgaddress{\street{250 Drillfield Drive}, \city{Blacksburg}, \postcode{24061}, \state{VA}, \country{USA}}}


\abstract{Gaussian processes (GPs) defined through intrinsic random fields provide a flexible framework for modeling spatial phenomena, and have been advocated in a variety of applications over the past several decades.
Nevertheless, their adoption has lagged behind traditional, covariance-based approaches, in part because the intrinsic formulation has lacked an accompanying toolkit of computational methods and dependence specifications that facilitate fitting and prediction. We  develop here a systematic framework for modeling intrinsic GPs and introduce practical algorithms and dependence/variogram models for modeling, inference and computation that parallel those of traditional, stationary GPs.  
We explore a close connection between intrinsic GP models and rational approximation, which clarifies the underlying problem structure 
Numerical examples illustrate how the new tools can be deployed in practice, highlighting the advantages of intrinsic-field modeling in terms of robustness, interpretability, and computational efficiency. }




\maketitle
 
\section{Introduction}
\label{sec:sec1}

Gaussian processes (GPs) have a long history in the interpolation and estimation of response surfaces over spatial domains \cite{matheron1973intrinsic,krige1976some,journel1976mining,cressie2015statistics}, and more recently, GPs have made inroads in machine learning, serving as surrogates for demanding computational models \cite{williams2006gaussian,gramacy2020surrogates}.   Surrogate GPs offer a flexible, explainable means to model, approximate, and quantify uncertainty of computer simulations.   However, as with any statistical model, different analytic assumptions of GPs may influence approaches for parameter estimation and, more importantly, final predictions.  For example, analyses often assume stationary GPs for computation and analytic convenience, but have (potentially) un-intentional influence on model predictions in unobserved regions of the response surface.  Joseph \cite{joseph2006limit} cited this influence and remedied the problem by  developing rational krigging.  In this paper,  we do the same.  We re-introduce {\em intrinsic} GPs (IGPs)  \cite{matheron1973intrinsic,besag1995bayesian} for modeling complex surfaces to avoid stationary GP pitfalls.  Additionally, we draw parallels between Joseph's \cite{joseph2006limit} work in rational interpolation and our IGP predictions a posteriori.

The primary difference between stationary GPs and intrinsic GPs (IGPs) is how spatial dependence is defined and estimated.  For example, consider a stationary GP (with constant mean $\mu$) 
$\{Z(\mathbf{s})|\mathbf{s}\in\Omega\}$ defined over the input space $\Omega$ (typically $\Omega = \mathbb{R}^p$).  Given any finite collection of points 
$\mathbf{s} = \{\mathbf{s}_1,\mathbf{s}_2,\ldots,\mathbf{s}_n\}\subset \Omega$, the random $n$-vector $\mathbf{Z}$ follows a multivariate normal distribution 
\[
\mathbf{Z} = (Z(\mathbf{s}_1),Z(\mathbf{s}_2), \ldots,Z(\mathbf{s}_n))^T \sim 
N_n(\mathbf{e} \mu,\boldsymbol{\Sigma})
\]
where $\mathbf{e}$ is the vector of ones and elements of $\boldsymbol{\Sigma}$, $\boldsymbol{\Sigma}_{ij}$,  are determined by a positive definite covariance function 
$\mathbf{C}(||\mathbf{s}_i-\mathbf{s}_i||)$.   This GP,  $\{Z(\mathbf{s})|\mathbf{s}\in\Omega\}$ , is termed stationary because the distribution of $\mathbf{Z}$ is unchanged when the locations $\mathbf{s}$  are displaced by a constant displacement vector $\mathbf{d}$
\[
\mathbf{Z}' = (Z(\mathbf{s}_1+\mathbf{d}),Z(\mathbf{s}_2+\mathbf{d}), \ldots,Z(\mathbf{s}_n+\mathbf{d}))^T \stackrel{d}{=}
\mathbf{Z}.
\]
Two common choices for the covariance function in stationary GPs are the {\em exponential}  
($\mathbf{C}(||\mathbf{s}_i-\mathbf{s}_i||)=
\sigma^2_z \exp\{-\theta||\mathbf{s}_i-\mathbf{s}_i||\}$) and the {\em Gaussian}  
($\mathbf{C}(||\mathbf{s}_i-\mathbf{s}_i||)=
\sigma^2_z \exp\{-\theta ||\mathbf{s}_i-\mathbf{s}_i||^2\}$). In both cases $\sigma_z^2$ controls the marginal variance of the process and $\theta$ controls how the spatial dependence decays as a function of distance.  Fig.~\ref{fig:fig1} shows several 1-d realizations of the exponential-covariance (top row, column 2) GP and Gaussian-covariance GP (bottom row, column 2), with mean $\mu=0$, $\sigma_z=1$, and $\theta=1$.

Unlike stationary GPs, Intrinsic GPs (IGPs) are characterized by variograms, $\gamma()$, not  covariance functions.  
Variograms express spatial dependence by quantifying the variance of process differences  so that $\gamma(\mathbf{s}, \mathbf{s}+\mathbf{d}) = \frac{1}{2} \mbox{Var}[
Z(\mathbf{s})-Z(\mathbf{s}+\mathbf{d})]$. Albeit, covariance functions of stationary GPs can  equivalently be represented with its variogram:
 $\gamma(\mathbf{s}, \mathbf{s}+\mathbf{d})=\gamma(\mathbf{0}, \mathbf{d}) = C(\mathbf{0}) - C(\mathbf{d})$,  but the reverse it not true for IGPs.   
For example, Brownian motion is a zeroth-order IGP with $\gamma(\mathbf{s}, \mathbf{s}+\mathbf{d}) = \sigma_z^2 ||\mathbf{d}||$ where $\sigma_z^2$ scales distances $\mathbf{d}$, and does not have a covariance function without added constraints ( e.g.,  $Z(0)=0$).  Similarly,  an IGP produced by convolving the Brownian motion process with a normal kernel that has standard deviation $r$ will not have a covariance function (without added constraints).  Even though this convolved-Brownian motion process produces smooth realizations, analogous to the realizations from a GP with the Gaussian covariance function, spatial dependence is only defined by its variogram.  The resulting variogram when  smoothing 1-dimensional Brownian motion is given by 
\[
\gamma(\mathbf{s}, \mathbf{s}+\mathbf{d}) = \frac{\sigma_z^2 }{2}
\left\{
  \frac{2r}{\sqrt{\pi}} \left[ \exp\left(
     \frac{-||\mathbf{d}||^2}{4r^2}\right) -1    \right] +
  ||\mathbf{d}|| \left[ \Phi\left(\frac{ ||\mathbf{d} ||}{\sqrt{2} r} \right) -
               \Phi\left(\frac{-||\mathbf{d} ||}{\sqrt{2} r} \right)
        \right]
\right\}
\]
where $\Phi(\cdot)$ is the standard normal cumulative density function (the generalization to $p$-dimensions is given in the appendix).

The variograms of Brownian motion IGP and convolved-Brownian motion IGP are plotted in column 1 of Fig.\ref{fig:fig1}.  In the same plots are the variograms translated from the exponential and Gaussian covariance functions.  Even though the IGP variograms are ``tuned'' to match those of their stationary counterparts at 0 and .05 (so that IGP samples will have similar small scale behavior as stationary GPs), 
there is no notion of a marginal mean nor variance in the IGP variograms.  On average, $[Z(\mathbf{s})-Z(\mathbf{s}+\mathbf{d})]^2$ grows as $||\mathbf{d}||$ grows in IGP variograms, whereas it approaches $2\sigma_z^2$ for large $||\mathbf{d}||$ in the stationary cases.   We observe these features in the stationary GP and  IGP random samples plotted in columns 2 and 3, respectively. (Note, to sample and plot IGP realizations in Fig.\ref{fig:fig1}, we initially force $Z(0)=0$ before producing the draws.  We then shift these draws so that most of them fit in the plotting region and can be compared visually to stationary GP samples. )
  
From plotting the IGP realizations, it is apparent that the large scale behavior of these IGPs differ from their stationary counterparts in that the intrinsic random trends can persist over long distances and there is no ``pull'' to an overall mean.  This results in noticeable differences  when we combine  IGPs with observations $Y$ in a statistical model to make predictions $Z(\mathbf{t})$ and at unobserved locations $\mathbf{t}$ a posteriori.
\begin{figure}[ht]
  \centerline{
   \includegraphics[width=5.2in,angle=0] {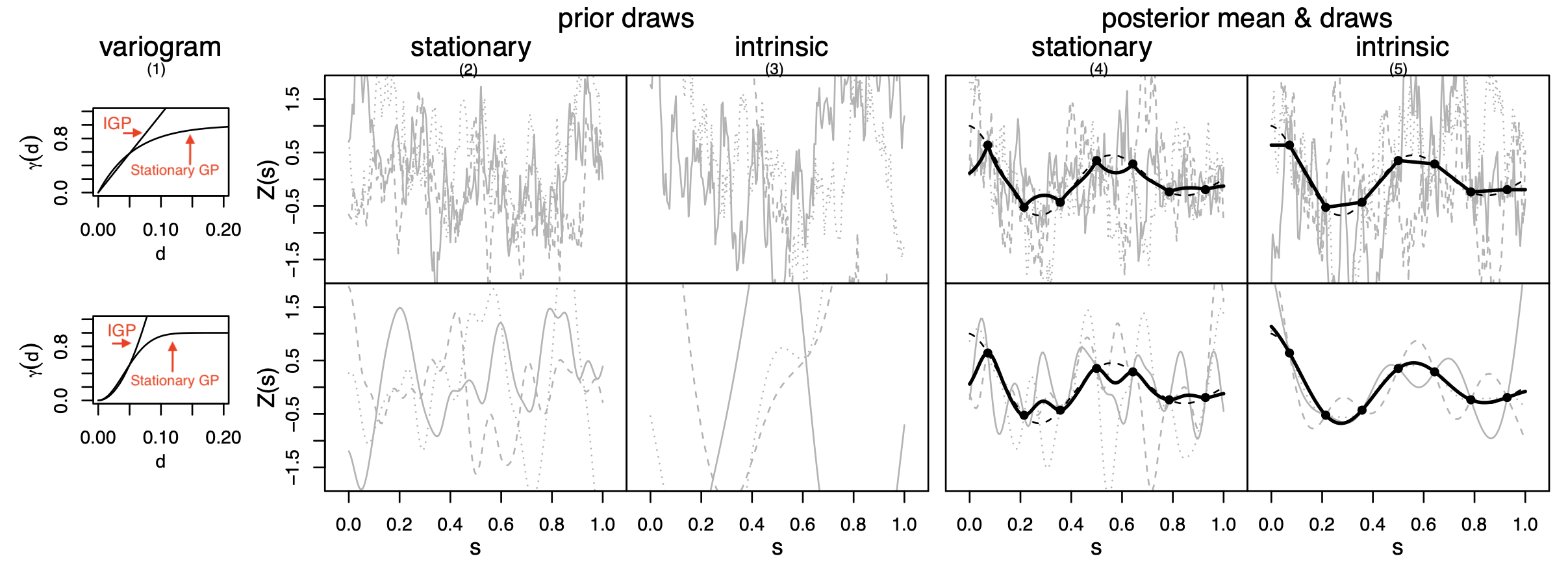}}
   \caption{\label{fig:fig1} Gaussian process-based interpolation of 1-d surfaces using  stationary and intrinsic GP formulations.  Variograms,  random samples, and posterior GP means given 7 data points (black dots) are plotted in column 1, 2-3, and 4-5 respectively.   The top row corresponds to GPs with rough, nowhere-differentiable realizations, one stationary and one intrinsic GP.  The bottom row corresponds to GPs that produce smooth realizations.  The stationary variogram parameters are specified to match the example in \cite{joseph2006limit}; the intrinsic variograms are specified to match the stationary variograms at 0 and .05. Notice posterior mean behavior between sampled points in columns 4 and 5.  Posterior model estimates (black lines) when assuming stationary GPs a priori (column 4) pull toward 0, the prior mean.  Whereas, posterior model estimates (black lines) when assuming IGPs a priori (column 5), do not.}
\end{figure}
  
As described in \sect \ref{sec:GPreg} and with using  bracket notation (defined in \cite{wikle2003hierarchical}), we take $[Z]$ to denote the prior (GP or IGP) for $Z(s)$; we take $[Y|Z]$ to denote the data model -- how the observations $Y$ are produced given $Z(s)$--, and apply  
Bayes rule to produce the posterior distribution for $Z$ given the observations $Y$: 
$[Z|Y] \propto [Y|Z][Z]$.  For simple examples, we plot realizations from a $[Z|Y]$ (gray lines) in  the $4^{\rm th}$ and $5^{\rm th}$ columns of Fig.\ref{fig:fig1} given 7 observations $Y$ (black dots). The posterior mean of $[Z|Y]$, sometimes called the {\em kriging} estimate, is drawn by the black lines in the 
$4^{\rm th}$ and $5^{\rm th}$ columns of Fig.\ref{fig:fig1}.  

Notice, all  of the posterior mean predictions in Fig.\ref{fig:fig1} interpolate the data $Y$.  However, those from the stationary GP formulations (4th column of Fig.\ref{fig:fig1} ) 
``sag'' towards zero, the prior mean of the GP between observations.  This is not always ideal in practice and  is due to specifying a stationary prior for $Z(s)$ with short range  spatial dependence.   Joseph  \cite{joseph2006limit} overcomes the sagging by introducing  rational Kriging \cite{joseph2025rational}; we do the same via assuming IGPs a priori.    Posterior process means derived from  IGP formulations do not revert to prior means.  Additionally, the approach we present maintains the benefits of modeling probabilistically in that our methods also result in joint, probabilistic
quantification of prediction uncertainty which allows for predictive simulations.  

 Inspired by Joseph's work (which we summarize and reframe into a  variogram setting in the next section, \sect \ref{sec:Josef}),  we seek in this paper to explore the connection between rational interpolation and posteriors resulting from ($0^{\rm th}$-order) IGP priors.    We draw from the classical literature on IGPs \cite{matheron1973intrinsic,journel1976mining,cressie2015statistics} as well as show important properties of variograms in \sect \ref{sec:VarioPriors} to discuss parallels between IGP models and rational estimation in \sect \ref{sec:GPreg}.  
This section also highlights features of GP regression with IGP priors and develops computational approaches for IGP simulation.  Even though IGPs have fewer parameters than stationary GPs, IGPs have shown themselves to be challenging to estimate and sample. Thus, we offer  computational guidelines for generating posterior means and realizations from  IGP-based posterior distributions in the context of simulating satellite data.    We conclude the paper with a  short discussion in \sect \ref{sec:summary}, and hope  to stimulate new directions and applications for IGPs and rational estimation.

\singlespacing

\section{Rational estimators recast with variograms}
\label{sec:Josef}

Simple kriging is an approach to GP regression that assumes a second‑order stationary GP model with a (known) constant mean, producing as a result a Best Linear Unbiased Predictor (BLUP). This predictor can show ``mean reversion.”  That is, away from data, predictions may be pulled toward the given global mean, which may distort the predictor significantly if the constant‑mean assumption is either wrong or mis-specified or likewise if the correlation parameters are mis-specified. This is particularly problematic when the underlying process has a trend, in which case assuming a global constant prior mean clearly will be unjustified.  To address these concerns, Joseph \cite{joseph2006limit} proposed modifying the predictor so that its local behavior is driven by nearby data rather than by a fixed prior mean.

Suppose for the moment that $Z(\mathbf{s})$ is stationary with a constant mean $\mu$. Define the standard deviation function for $Z(\mathbf{s})$ as $\varsigma(\mathbf{s})=\sqrt{\mathbb{E}[(Z(\mathbf{s})-\mu)^{2}]}$, and the correlation function as $R(\mathbf{s},\mathbf{t})=\mathbb{E}[(Z(\mathbf{s})-\mu)(Z(\mathbf{t})-\mu)]/(\varsigma(\mathbf{s})\varsigma(\mathbf{t}))$. 
Starting from simple Kriging with an arbitrarily assigned prior mean, \cite{joseph2006limit} defines a recursive predictor in which the ``mean" used at step $\ell+1$ is the predicted mean obtained from step $\ell$. This is similar in spirit (across a single step) to ``ordinary Kriging."  However, Joseph continues the recursion to a limit and derives  asymptotically the ``limit-Kriging" predictor, $\widehat{Z}(\mathbf{t})$: 
$$
 \widehat{Z}(\mathbf{t})=\sum_{k=1}^n \lambda_k^{\mathsf{lim}}(\mathbf{t})y_k
 \quad\mbox{with}\quad
\lambda_k^{\mathsf{lim}}(\mathbf{t})
=\frac{\mathbf{r}(\mathbf{t})^\top\mathbf{R}^{-1}\mathbf{e}_k}{\mathbf{r}(\mathbf{t})^\top\mathbf{R}^{-1}\mathbf{e}}
$$
where $y_k$ identifies observation at location $\mathbf{s}_k$ for $k\in \{1,,,n\}$; $\lambda_k^{\mathsf{lim}}(\mathbf{t})$ is the $k^{th}$ limit-Kriging coefficient; $\mathbf{e}$ is the vector of ones; and 
the components of the correlation matrix,
$\mathbf{R}\in\mathbb{R}^{n\times n}$, and the correlation vector, $\mathbf{r}(\mathbf{t})\in\mathbb{R}^{n}$, are defined in terms of the correlation function:
$\mathbf{R}_{ij}=R(\mathbf{s}_i,\mathbf{s}_j)$ and $\mathbf{r}_i(\mathbf{t})=R(\mathbf{s}_i,\mathbf{t})$. 

This predictor,  $\widehat{Z}(\mathbf{t})$, can be viewed as a ``rational function" with respect to  correlations between $Z(\mathbf{t})$ and each of $Z(\mathbf{s}_i)$ for $i=1,\ldots,n$. Indeed, the limit-Kriging predictor is precisely the ratio of two linear functions of those correlation weights, and if one allows cross-correlations to weaken, the predicted values will tend toward the value at the nearest observation point rather than to the global mean.

We reframe these ideas into terms that are sensible for intrinsic random fields, including IGPs. Noting that 
$$
\gamma(\mathbf{s},\mathbf{t})=\frac12\mathbb{E}[(Z(\mathbf{s})-Z(\mathbf{t}))^{2}]=\frac12\varsigma(\mathbf{s})^{2}+ \frac12\varsigma(\mathbf{t})^{2}-\mathbb{E}[(Z(\mathbf{s})-\mu)(Z(\mathbf{t})-\mu)],
$$
we may express the correlation function in terms of the variogram function: 
$$
R(\mathbf{s},\mathbf{t})=\frac12\frac{\varsigma(\mathbf{s})^{2}+\varsigma(\mathbf{t})^{2}}{\varsigma(\mathbf{s})\varsigma(\mathbf{t})}-\frac{\gamma(\mathbf{s},\mathbf{t})}{\varsigma(\mathbf{s})\varsigma(\mathbf{t})}.
$$
Joseph assumes the correlation function is given. Standard deviations, $\{\varsigma(\mathbf{t}),\varsigma(\mathbf{s}_1),\varsigma(\mathbf{s}_2),\ldots,\varsigma(\mathbf{s}_n)\}$are either known or remain as parameters to be estimated.  We assume access to a variogram function, $\gamma(\mathbf{s},\mathbf{t}).$ To enable comparisons, we assume that both the random field mean and variance are constant across $\Omega$. Without loss of generality, take $\mu=0$ and 
$\varsigma(\mathbf{t})=\varsigma(\mathbf{s}_1)=\varsigma(\mathbf{s}_2)=\cdots=\varsigma(\mathbf{s}_n)=\vartheta$.   With these substitutions, we express the limit-Kriging coefficients in terms of the variogram function:
\begin{equation}
    \label{limKrigCoeff}
\displaystyle
\lambda_k^{\mathsf{lim}}(\mathbf{t})
=\frac{(\vartheta^{2}\mathbf{e}-\boldsymbol{\gamma}(\mathbf{t}))^\top(\vartheta^{2}\mathbf{e}\mathbf{e}^\top-\boldsymbol{\Gamma})^{-1}\mathbf{e}_k}{(\vartheta^{2}\mathbf{e}-\boldsymbol{\gamma}(\mathbf{t}))^\top(\vartheta^{2}\mathbf{e}\mathbf{e}^\top-\boldsymbol{\Gamma})^{-1}\mathbf{e}}
\end{equation}
The limit-Kriging coefficients have a functional form reflecting rational dependence on the core variogram location parameters $\{\boldsymbol{\gamma}(\mathbf{t},\mathbf{s}_1),\boldsymbol{\gamma}(\mathbf{t},\mathbf{s}_2),\ldots,\boldsymbol{\gamma}(\mathbf{t},\mathbf{s}_n)\}$. They may be viewed, in particular,
 as a ratio of linear functions of these  parameters. 
  
In subsequent work, Joseph and Kang \cite{joseph2011regression} approach  computational fragility of Kriging in large‑$n$ (many observations) and large‑$p$ (many prediction locations) settings by grafting a regression framework that captures trends onto Shepard's general barycentric rational-form interpolation \cite{shepard1968two}.  The key component of Shepard interpolation is a distance measure $\mathbbm{d}(\mathbf{s},\mathbf{t})$ (nominally a metric) and a positive weight vector $\mathbf{c}\in\mathbb{R}^n$.   
The Shepard interpolant is then 
$
\displaystyle S(\mathbf{t})=\left.{\displaystyle \sum_{k=1}^n\frac{c_k\, y_k}{\mathbbm{d}(\mathbf{t},\mathbf{s}_k)}}\middle\slash{\displaystyle \sum_{j=1}^n\frac{c_j}{\mathbbm{d}(\mathbf{t},\mathbf{s}_j)}}\right.
$.
So long as $\mathbbm{d}(\mathbf{s},\mathbf{t})$ varies continuously with respect to $\mathbf{s},\mathbf{t}\in \mathbb{R}^d$, we have that $S(\mathbf{t})\rightarrow y_{\ell}$ as 
$\mathbbm{d}(\mathbf{t},\mathbf{s}_{\ell})\rightarrow 0$.  The Shepard interpolant is also known as an \emph{Inverse Distance Weighted} (IDW) interpolant; it has a barycentric form with respect to the separation parameters $\{\mathbbm{d}(\mathbf{t},\mathbf{s}_{1}),\mathbbm{d}(\mathbf{t},\mathbf{s}_{2}),\ldots,\mathbbm{d}(\mathbf{t},\mathbf{s}_{n})\}$: and so is in a sense in ``rational form."   Furthermore, if the distance measure $\mathbbm{d}(\mathbf{s},\mathbf{t})$ is translation invariant, so that $\mathbbm{d}(\mathbf{s},\mathbf{t})=\mathbbm{d}(\mathbf{s}+\mathbf{h},\mathbf{t}+\mathbf{h})$ for all $\mathbf{h}\in\mathbb{R}^d$, then $S(\mathbf{t})$ will be insensitive to data components related to trends and drift. For example, if $Z(\mathbf{t})$ has a trend so that  $Z(\mathbf{t})-Z(\mathbf{t}+\mathbf{h})\sim N(\Delta(\mathbf{h}),\mathbf{G}\mathbf{G}^\top)$ then $y_k+\Delta(\mathbf{h})-Z(\mathbf{s}_k+\mathbf{h})\sim N(0,\sigma^2\mathbf{F}\mathbf{F}^\top)$ and $\mathbb{E}[S(\mathbf{t})]=\mathbb{E}[S(\mathbf{t}+\mathbf{h})]$ with a similar translation invariance holding for the covariance. 

Starting from Shepard interpolation, Joseph and Kang \cite{joseph2011regression} add a polynomial regressor, $p(\mathbf{t})$, to capture trends and include an anisotropic scaling and exponential distance measure to reflect variable importance and  improve numerical behavior. Their resulting predictor is essentially a rationally normalized kernel smoother built around a regression fit.  
$$
 \widehat{Z}(\mathbf{t})=p(\mathbf{t})+\sum_{k=1}^n \lambda_k^{\mathsf{shep}}(\mathbf{t})(y_k-p(\mathbf{s}_k))
 \quad\mbox{with}\quad
\lambda_k^{\mathsf{shep}}(\mathbf{t})
=\left.{\displaystyle \frac{c_k}{\mathbbm{d}(\mathbf{t},\mathbf{s}_k)}}\middle\slash{\displaystyle \sum_{j=1}^n\frac{c_j}{\mathbbm{d}(\mathbf{t},\mathbf{s}_j)}}\right..
$$

While \cite{joseph2011regression} uses a weighted Euclidean length as a distance measure for their Shepard interpolation component, many other possibilities are evident.  The variogram function, $\gamma(\mathbf{s},\mathbf{t})$ will generally fail to satisfy the conditions of a metric, however one may still consider using it in a variant Shepard-style interpolant defined as:
$$
S_{\gamma}(\mathbf{t})=\left.{\displaystyle \sum_{k=1}^n\frac{c_k\, y_k}{\gamma(\mathbf{t},\mathbf{s}_k)}}\middle\slash{\displaystyle \sum_{j=1}^n\frac{c_j}{\gamma(\mathbf{t},\mathbf{s}_j)}}\right..
$$
 $S_{\gamma}(\mathbf{t})$ shares the same insensitivity to trends and drift as does the Shepard interpolant with  translation invariant metric.  While introducing the variogram function as a surrogate distance measure may make $S_{\gamma}$ a natural candidate to use for  prediction in an intrinsic GP setting, our focus will shift in a different direction.  

Notice that the stated motivation of \cite{joseph2011regression} is different from that of the earlier work \cite{joseph2006limit} in that  \cite{joseph2011regression} focusses  on computational issues  arising in Kriging, especially for large numbers of observations or  large numbers of predictors. Joseph and Kang's proposed solution in \cite{joseph2011regression} is built around a rational-form Shepard-style interpolant, which they find to be a scalable alternative to traditional Kriging though it takes quite a different form from the earlier limit-Kriging of \cite{joseph2006limit}. Nonetheless, there are similarities in the two approaches: interpolation of regression residuals used to capture trends in \cite{joseph2011regression} has a similar flavor to limit-Kriging’s ``localization” goal in \cite{joseph2006limit}. Also, the purely distance‑based, Shepard-style interpolation of residuals in \cite{joseph2011regression} does not contain any global mean to which predictions can revert in sparse-data regions. Thus, residual correction is entirely controlled by Shepard-style interpolation and not by an implicitly stationary GP.  As a consequence, away from data the predictor will tend toward the regression surface rather than toward a global mean.

Joseph in \cite{joseph2025rational} unifies and extends ideas presented in \cite{joseph2006limit} and \cite{joseph2011regression}, introducing an umbrella approach he calls ``rational Kriging."  He allows a nonstationary covariance 
and seeks a minimal mean-square-loss linear predictor under the constraint that the coefficients sum to one.  
The derived Kriging predictor is a  a convex combination of observations and has a rational form with respect to kernel evaluations — essentially the ratio of two kernel expansions.   

The ``rational Kriging" predictor is defined as
$$
\widehat{Z}(\mathbf{t})=\sum_{k=1}^n \lambda_k^{\mathsf{rat}}(\mathbf{t})y_k
 \quad\mbox{with}\quad\lambda_k^{\mathsf{rat}}(\mathbf{t})
=\frac{\mathbf{r}(\mathbf{t})^\top\mathbf{R}^{-1}\mathbf{e}_k (\mathbf{e}_k^\top\mathbf{R}\mathbf{c}) }{\mathbf{r}(\mathbf{t})^\top\mathbf{c}}
$$
where $\mathbf{c}> \mathbf{0}$ elementwise. Joseph considers a variety of choices for $\mathbf{c}$, among them the Perron eigenvector for $\mathbf{R}$. Choosing $\mathbf{c}=\mathbf{R}^{-1}\mathbf{e}$ recovers the limit-Kriging predictor, though in this case one cannot guarantee in advance that $\mathbf{c}>\mathbf{0}$ elementwise. 
All three families of Kriging coefficients produce predictors that are affine combinations of observations: $\displaystyle \sum_{k=1}^n\lambda_k^{\mathsf{lim}}(\mathbf{t})=\sum_{k=1}^n\lambda_k^{\mathsf{shep}}(\mathbf{t})=\sum_{k=1}^n\lambda_k^{\mathsf{rat}}(\mathbf{t})=1$. Both the Shepard/IDW and the rational Kriging coefficients have the added property of nonnegativity, $\lambda_k^{\mathsf{shep}}(\mathbf{t})\geq 0$  and $\lambda_k^{\mathsf{rat}}(\mathbf{t})\geq 0$, so both of the associated  predictors yield convex combinations of observations.  Recasting the rational Kriging coefficients in terms of the variogram function, we find:
$$
\displaystyle
\lambda_k^{\mathsf{rat}}(\mathbf{t})
=\frac{(\vartheta^{2}\mathbf{e}-\boldsymbol{\gamma}(\mathbf{s}_k))^\top\mathbf{c}}{(\vartheta^{2}\mathbf{e}-\boldsymbol{\gamma}(\mathbf{t}))^\top\mathbf{c}}\,
\left(\vartheta^{2}\mathbf{e}-\boldsymbol{\gamma}(\mathbf{t})\right)^\top\left(\vartheta^{2}\mathbf{e}\mathbf{e}^\top-\boldsymbol{\Gamma}\right)^{-1}\mathbf{e}_k. 
$$ 
We fix the variance at $\vartheta^{2}=1$ and write  $R(\mathbf{s},\mathbf{t})=
1-\boldsymbol{\gamma}(\mathbf{s},\mathbf{t})$.  
Stationarity implies that the variogram has a sill at $\vartheta^{2}=1$, and under these working assumptions the variogram function will saturate at the sill, $\gamma(\mathbf{s},\mathbf{t})\approx 1$, with large separation distances, $\|\mathbf{s}-\mathbf{t}\|\gg 0$.
This motivates introducing a surrogate variogram function, $\hat{\boldsymbol{\gamma}}(\mathbf{s},\mathbf{t})$, that will \emph{not} have a sill but rather is designed so as to increase the dynamic range of $\boldsymbol{\gamma}$ in the sense that $\hat{\boldsymbol{\gamma}}$ will reflect similar behaviour to that of  $\boldsymbol{\gamma}$ at small separation distances (in which case $\hat{\boldsymbol{\gamma}}\approx\boldsymbol{\gamma}$), while not saturating at large separation distances. To that end, we define $\hat{\boldsymbol{\gamma}}(\mathbf{s},\mathbf{t})=\varrho\frac{\boldsymbol{\gamma}(\mathbf{s},\mathbf{t})}{1-\boldsymbol{\gamma}(\mathbf{s},\mathbf{t})}$, where $\varrho>0$ is a scale factor related to the rate of growth of $\frac{\boldsymbol{\gamma}(\mathbf{s},\mathbf{t})}{1-\boldsymbol{\gamma}(\mathbf{s},\mathbf{t})}$ as the separation distance, $\|\mathbf{s}-\mathbf{t}\|$ grows. $\hat{\boldsymbol{\gamma}}(\mathbf{s},\mathbf{t})$ will range across $[0,\infty)$ as $\|\mathbf{s}-\mathbf{t}\|$ increases.

 We wish to consider the behaviour of the estimator as $\varrho$ varies and becomes small, so we will assume that $\hat{\boldsymbol{\gamma}}(\mathbf{s},\mathbf{t})$  varies independently of $\varrho$.
One may observe immediately that as a consequence, the variogram function, $\boldsymbol{\gamma}(\mathbf{s},\mathbf{t})$, and the correlation function, $R(\mathbf{s},\mathbf{t})$, will both vary with $\varrho$, and 
$$
\boldsymbol{\gamma}(\mathbf{s},\mathbf{t})=\frac{\frac1{\varrho}\hat{\boldsymbol{\gamma}}(\mathbf{s},\mathbf{t})}{1+\frac1{\varrho}\hat{\boldsymbol{\gamma}}(\mathbf{s},\mathbf{t})}
\quad\mbox{and}\quad
R(\mathbf{s},\mathbf{t})=
1-\boldsymbol{\gamma}(\mathbf{s},\mathbf{t})=\frac1{1+\frac1{\varrho}\hat{\boldsymbol{\gamma}}(\mathbf{s},\mathbf{t})}.
$$ 

This reformulation amounts to a  change in the variogram/correlation model that reveals a ``rational" dependence between $\hat{\boldsymbol{\gamma}}(\mathbf{s}_i,\mathbf{t})$ --- a dependence that can be framed in barycentric form.  To see this note first that  
$$
\mathbf{R}_{ij}
=\frac1{1+\frac1{\varrho}\hat{\boldsymbol{\gamma}}(\mathbf{s}_i,\mathbf{s}_j)}
\quad \mbox{and}\quad 
\mathbf{r}_i(\mathbf{t})=\frac1{1+\frac1{\varrho}\hat{\boldsymbol{\gamma}}(\mathbf{t},\mathbf{s}_i)}
$$
Hence for small values of $\varrho$, $\mathbf{R}=\mathbf{I}+\mathcal{O}(\varrho)$,   
$$
\frac{\mathbf{r}_i(\mathbf{t})}{\mathbf{r}(\mathbf{t})^\top\mathbf{c}}=\frac{\frac1{\varrho+\hat{\boldsymbol{\gamma}}(\mathbf{t},\mathbf{s}_i)}}{\displaystyle \sum_{\ell=1}^n\frac{c_{\ell}}{\varrho+\hat{\boldsymbol{\gamma}}(\mathbf{t},\mathbf{s}_i)}},
\quad\mbox{and}\quad
\mathbf{e}_k^\top\mathbf{R}\mathbf{c}=c_k+\varrho\sum_{\substack{\ell=1\\\ell\neq k}}^n\frac{c_{\ell}}{\varrho+\hat{\boldsymbol{\gamma}}(\mathbf{s}_k,\mathbf{s}_i)}.
$$
One may recast Joseph's rational Kriging coefficients  in terms of the surrogate variogram, $\hat{\boldsymbol{\gamma}}$: 
$$
\lambda_k^{\mathsf{rat}}(\mathbf{t})
=\frac{\mathbf{r}(\mathbf{t})^\top\mathbf{R}^{-1}\mathbf{e}_k (\mathbf{e}_k^\top\mathbf{R}\mathbf{c}) }{\mathbf{r}(\mathbf{t})^\top\mathbf{c}}
=\frac{\frac{c_k}{\varrho+\hat{\boldsymbol{\gamma}}(\mathbf{t},\mathbf{s}_k)}+\mathcal{O}(\varrho)}{\displaystyle \sum_{\ell=1}^n\frac{c_{\ell}}{\varrho+\hat{\boldsymbol{\gamma}}(\mathbf{t},\mathbf{s}_{\ell})}}\quad\stackrel{\varrho\rightarrow 0}{\longrightarrow} \quad \frac{\frac{c_k}{\hat{\boldsymbol{\gamma}}(\mathbf{t},\mathbf{s}_k)}}{\displaystyle \sum_{\ell=1}^n\frac{c_{\ell}}{\hat{\boldsymbol{\gamma}}(\mathbf{t},\mathbf{s}_{\ell})}},
$$
so that in the limit of $\varrho\rightarrow 0$, the rational Kriging estimator approaches a form of Shepard/IDW Interpolation: 
$$
\widehat{Z}(\mathbf{t})=\left.{\displaystyle \sum_{k=1}^n\frac{c_k\ y_k}{\hat{\boldsymbol{\gamma}}(\mathbf{t},\mathbf{s}_k)}}\middle\slash{\displaystyle \sum_{\ell=1}^n\frac{c_{\ell}}{\hat{\boldsymbol{\gamma}}(\mathbf{t},\mathbf{s}_{\ell})}}\right.
$$

Joseph's ideas, recast with variograms in this way, illuminate connections with approaches that utilize prior IGPs.   In particular, we find in \sect \ref{sec:funcKrig} useful similarities in the Kriging coefficient forms. 

\section{Variogram Properties}
\label{sec:VarioPriors}

In passing from rational Kriging  described in \sect \ref{sec:Josef} to IGP regression in \sect \ref{sec:GPreg}, we highlight helpful properties of variograms and the covariances they induce. 

\subsection{Variograms are conditionally negative definite}

If $\{Z(\mathbf{s})|\mathbf{s}\in\Omega\}$ is a random field on $\Omega\subset\mathbb{R}^d$, the corresponding \textbf{{(semi-)} variogram} is a function $\gamma:\Omega\times\Omega\rightarrow \mathbb{R}$ defined as $\gamma(\mathbf{s},\mathbf{t})=\frac12\mathbb{E}[(Z(\mathbf{s})-Z(\mathbf{t}))^2]$. If  $\gamma(\mathbf{s},\mathbf{t})$ is jointly continuous in $\mathbf{s}$ and $\mathbf{t}$ then $Z(\mathbf{s})$ is said to be mean-square continuous. 
A key property of variograms is conditional negative-definiteness: 
\begin{definition}[]
Suppose that $\Omega$ is a domain in $\mathbb{R}^d$, and let $g:\Omega\times\Omega\rightarrow \mathbb{R}$
 be a given function.  $g$ is said to be \emph{conditionally negative-definite} if 
     for all choices of 
    $\{\mathbf{s}_1,\mathbf{s}_2,\ldots,\mathbf{s}_n\}\subset \Omega$ and scalars $\{\alpha_1,\alpha_2,\ldots,\alpha_n\}\subset \mathbb{R}$ such that $\displaystyle \sum_{k=1}^n\alpha_k=0$, 
    it happens that
    $\displaystyle \sum_{k,\ell=1}^n\alpha_k\,\alpha_\ell \,g(\mathbf{s}_k,\mathbf{s}_{\ell})\leq 0$.  
\end{definition}

\begin{theorem}
Variograms are conditionally negative definite.  Furthermore, for any  $\mathbf{t}\in\Omega$ distinct from a fixed set of distinct points, $\{\mathbf{s}_1,\mathbf{s}_2,\ldots,\mathbf{s}_n\}\subset \Omega$, suppose 
{\small
\begin{equation}
    \label{varDef}
\boldsymbol{\gamma}(\mathbf{t})=\begin{pmatrix}
\gamma(\mathbf{s}_1,\mathbf{t})\\
\gamma(\mathbf{s}_2,\mathbf{t})\\
\vdots\\
\gamma(\mathbf{s}_n,\mathbf{t})
\end{pmatrix},\ 
\mathbf{e}=\begin{pmatrix}
1\\ 1\\ \vdots\\ 1 \end{pmatrix}\in\mathbb{R}^{n},
\ \mbox{and} \ 
\boldsymbol{\Gamma}=\begin{bmatrix}
0 & \gamma(\mathbf{s}_1,\mathbf{s}_2) &\ldots & \gamma(\mathbf{s}_1,\mathbf{s}_n)\\
\gamma(\mathbf{s}_2,\mathbf{s}_1) & 0 &\ldots & \gamma(\mathbf{s}_2,\mathbf{s}_n)\\
\vdots & & \ddots & \vdots\\
\gamma(\mathbf{s}_n,\mathbf{s}_1) & \gamma(\mathbf{s}_n,\mathbf{s}_2) &\ldots & 0
\end{bmatrix}.
\end{equation}
}
Then the matrix $\boldsymbol{\gamma}(\mathbf{t})\,\mathbf{e}^{\!\top}+\mathbf{e}\,\boldsymbol{\gamma}(\mathbf{t})^{\!\top}-\boldsymbol{\Gamma}$ is positive semi-definite. 
 \end{theorem}
\noindent
\emph{Proof:}
For conciseness, denote $b_k=Z(\mathbf{s}_k)$ for $k=1,\ldots,n$ and choose scalars $\{\alpha_1,\alpha_2,\ldots,\alpha_n\}\subset \mathbb{R}$ such that  $ \sum_{k=1}^n\alpha_k=0$. We have then
\begin{align*}
\sum_{i,j=1}^n\alpha_i\,\alpha_j \,&\gamma(\mathbf{s}_i,\mathbf{s}_j)= \frac12 \sum_{i,j=1}^n\alpha_i\,\alpha_j \,\mathbb{E}[\left((b_i-b_1)-(b_j-b_1)\right)^2]\\
&= \frac12 \underbrace{\sum_{i,j=1}^n\alpha_i\,\alpha_j \,\mathbb{E}[(b_i-b_1)^2]}_{=0} +\frac12 \underbrace{\sum_{i,j=1}^n\alpha_i\,\alpha_j \,\mathbb{E}[(b_j-b_1)^2]}_{=0}\\
&\qquad\qquad -\sum_{i,j=2}^n\alpha_i\,\alpha_j \, \mathbb{E}[(b_i-b_1)(b_j-b_1)]\\
&= -\mathbb{E}\left[\sum_{i=2}^n\alpha_i\, (b_i-b_1)\cdot \sum_{j=2}^n\alpha_j(b_j-b_1)\right]=-\mathbb{E}\left[\left|\sum_{i=2}^n\alpha_i(b_i-b_1)\right|^2\right]\leq 0,
\end{align*}
which establishes the first claim.  Refering to quantities defined in \eqref{varDef},
the conditional negative semidefiniteness of $\gamma$ can be stated concisely as $\mathbf{a}^\top\boldsymbol{\Gamma}\mathbf{a}\leq 0$ whenever $\mathbf{a}\in\mathbb{R}^n$ satisfies $\mathbf{a}^\top\mathbf{e}=0$. Notice that  $\begin{bmatrix}
    0 & \boldsymbol{\gamma}(\mathbf{t})^{\!\top} \\
    \boldsymbol{\gamma}(\mathbf{t}) & \boldsymbol{\Gamma}
\end{bmatrix}$ is a variogram matrix evaluated on an augmented set of points $\{\mathbf{t},\mathbf{s}_1,\mathbf{s}_2,\ldots,\mathbf{s}_n\}\subset \Omega$.
As a consequence, we have $\mathbf{a}^\top\begin{bmatrix}
    0 & \boldsymbol{\gamma}(\mathbf{t})^{\!\top} \\
    \boldsymbol{\gamma}(\mathbf{t}) & \boldsymbol{\Gamma}
\end{bmatrix}\mathbf{a}\leq 0$ whenever $\mathbf{a}\in\mathbb{R}^{n+1}$ satisfies $\mathbf{a}^\top\begin{pmatrix}
    1\\ \mathbf{e}
\end{pmatrix}=0$. 
Suppose that $\mathbf{x}\in\mathbb{R}^n$ is chosen arbitrarily and for $\mathbf{D}=\left[\begin{array}{ccccc} -1&1 & 0 &\cdots& 0\\ -1&0 &1 &  & 0 \\
\vdots & & & \ddots & \\ -1&0 & 0 &\cdots& 1\end{array}\right]\in\mathbb{R}^{n\times(n+1)}$ define   $\mathbf{a}=\mathbf{D}^\top\mathbf{x}=\begin{pmatrix}
    -\mathbf{e}^\top\mathbf{x}\\ \mathbf{x}
\end{pmatrix}$. Then clearly $\mathbf{a}^\top\begin{pmatrix}
    1\\ \mathbf{e}
\end{pmatrix}=0$ and we may conclude that
$$
\mathbf{x}^\top(\boldsymbol{\Gamma}-\boldsymbol{\gamma}(\mathbf{t})\,\mathbf{e}^{\!\top}-\mathbf{e}\,\boldsymbol{\gamma}(\mathbf{t})^{\!\top})\mathbf{x}=\mathbf{x}^\top\mathbf{D}\begin{bmatrix}
    0 & \boldsymbol{\gamma}(\mathbf{t})^{\!\top} \\
    \boldsymbol{\gamma}(\mathbf{t}) & \boldsymbol{\Gamma}
\end{bmatrix}\mathbf{D}^T\mathbf{x}=\mathbf{a}^\top\begin{bmatrix}
    0 & \boldsymbol{\gamma}(\mathbf{t})^{\!\top} \\
    \boldsymbol{\gamma}(\mathbf{t}) & \boldsymbol{\Gamma}
\end{bmatrix}\mathbf{a}\leq 0,
$$ 
and so, $\boldsymbol{\gamma}(\mathbf{t})\,\mathbf{e}^{\!\top}+\mathbf{e}\,\boldsymbol{\gamma}(\mathbf{t})^{\!\top}-\boldsymbol{\Gamma}=-\mathbf{D}\begin{bmatrix}
    0 & \boldsymbol{\gamma}(\mathbf{t})^{\!\top} \\
    \boldsymbol{\gamma}(\mathbf{t}) & \boldsymbol{\Gamma}
\end{bmatrix}\mathbf{D}^\top$ is positive semidefinite. $\Box$

\bigskip
The matrix $\boldsymbol{\gamma}(\mathbf{t})\,\mathbf{e}^{\!\top}+\mathbf{e}\,\boldsymbol{\gamma}(\mathbf{t})^{\!\top}-\boldsymbol{\Gamma}$ is of particular interest because it is the covariance of the vector of increments between observation locations and the given target.\\
That is, with the $\mathbf{D}$ defined above, for  $\mathbf{D}\,\mathbb{Z}=\begin{pmatrix}
        Z(\mathbf{s}_1)-Z(\mathbf{t})\\
        Z(\mathbf{s}_2)-Z(\mathbf{t})\\ \vdots \\
        Z(\mathbf{s}_n)-Z(\mathbf{t})
 \end{pmatrix}$ 
 and
 $\mathbb{Z}=\begin{pmatrix}
        Z(\mathbf{t})\\ Z(\mathbf{s}_1)\\Z(\mathbf{s}_2)\\ \vdots \\Z(\mathbf{s}_n)
\end{pmatrix},
$
we have that
$
\boldsymbol{\gamma}(\mathbf{t})\,\mathbf{e}^{\!\top}+\mathbf{e}\,\boldsymbol{\gamma}(\mathbf{t})^{\!\top}-\boldsymbol{\Gamma} =\mathsf{cov}(\mathbf{D}\,\mathbb{Z}).$
This equality can be seen after one notes first an algebraic identity: 
{\small 
$$
(Z(\mathbf{s}_i)-Z(\mathbf{t}))(Z(\mathbf{s}_j)-Z(\mathbf{t}))=\frac12\left[(Z(\mathbf{s}_i)-Z(\mathbf{t}))^{2}+(Z(\mathbf{s}_j)-Z(\mathbf{t}))^{2}-(Z(\mathbf{s}_i)-Z(\mathbf{s}_j))^{2} \right]
$$ }
then take expectations: 
{\small 
$
\mathbb{E}[(Z(\mathbf{s}_i)-Z(\mathbf{t}))(Z(\mathbf{s}_j)-Z(\mathbf{t}))]=\frac12\left[\gamma(\mathbf{s}_i,\mathbf{t})+\gamma(\mathbf{s}_j,\mathbf{t})-\gamma(\mathbf{s}_i,\mathbf{s}_j) \right].
$}  
\begin{theorem} \label{ThmCondNegDef}
   Choose locations $\{\mathbf{s}_1,\mathbf{s}_2,\ldots,\mathbf{s}_n\}\subset\Omega$ arbitrarily. Let $\boldsymbol{\Gamma}$ be the   $n\times n$ matrix having entries $\Gamma_{ij}=\gamma(\mathbf{s}_i,\mathbf{s}_j)$. Assume that $\boldsymbol{\Gamma}$ is nonsingular.
Then we have 
\begin{enumerate}
    \item[(a)] $\boldsymbol{\Gamma}$ has \emph{only one} positive eigenvalue. All remaining eigenvalues are strictly negative.
    \item[(b)] The eigenvector, $\mathbf{u}_+$, associated with the single  positive eigenvalue has single signed entries and so may be assumed to be a strictly element-wise positive vector.  
    \item[(c)] There exists $\delta>0$ sufficiently large so that $\delta\mathbf{e}\mathbf{e}^\top-\boldsymbol{\Gamma}$ is positive definite. 
\end{enumerate}
\end{theorem}

\emph{Proof:} Note first that $\boldsymbol{\Gamma}$ is a real symmetric matrix so all its eigenvalues are real; label them in decreasing order as $\varsigma_1\geq\varsigma_2\geq\cdots\geq\varsigma_n$.
Since $\gamma(\mathbf{s},\mathbf{t})>0$ for all $\mathbf{s}\neq\mathbf{t}$,
$\boldsymbol{\Gamma}$ is an irreducible, elementwise nonnegative matrix, and the Perron-Frobenius Theorem \cite[Theorem 8.4.4, p. 534]{HornJohnson2013MatrixAnalysis} guarantees the spectral radius of the matrix is itself a simple eigenvalue  with an eigenvector having single-signed elements that we may assume to be  strictly positive.  In particular,  we must have $\varsigma_1>0$.

Conditional negative definiteness of the variogram means for $\boldsymbol{\Gamma}$ that if $\mathbf{e}=[1,1,\ldots,1]^T\in\mathbb{R}^{n}$ and $\mathbf{a}=[\alpha_1,\alpha_2,\ldots,\alpha_n]^T\in\mathbb{R}^{n}$ then whenever  $\mathbf{e}\perp \mathbf{a}$ (equivalent to $\sum_{k=1}^n\alpha_k=0$) we must have that $\mathbf{a}^\top\boldsymbol{\Gamma}\mathbf{a}\leq 0$. 
The Courant-Fisher variational characterization of the eigenvalues of $\boldsymbol{\Gamma}$ \cite[Theorem 4.2.6, p. 236]{HornJohnson2013MatrixAnalysis}  asserts  in general,that  one has $\displaystyle \varsigma_k=\min_{\dim{\mathfrak{P}}=k-1}\max_{\mathbf{v}\in\mathfrak{P}^{\perp}}\frac{\mathbf{v}^T\boldsymbol{\Gamma}\mathbf{v}}{\|\mathbf{v}\|^{2}}$. In particular for $k=2$, 
$$
\displaystyle \varsigma_2=\min_{\dim{\mathfrak{P}}=1}\max_{\mathbf{v}\in\mathfrak{P}^{\perp}}\frac{\mathbf{v}^T\boldsymbol{\Gamma}\mathbf{v}}{\|\mathbf{v}\|^{2}}\leq \max_{\mathbf{v}\perp\mathbf{e}}\frac{\mathbf{v}^T\boldsymbol{\Gamma}\mathbf{v}}{\|\mathbf{v}\|^{2}}\leq 0.$$ Nonsingularity of $\boldsymbol{\Gamma}$ guarantees $0\neq\varsigma_2$, hence $0>\varsigma_2\geq\cdots\geq\varsigma_n$.  Thus,   assertion (a) is true, and since $\varsigma_1$ is the only positive eigenvalue, assertion (b) is true as well. 

For assertion (c), note that $(\lambda,\mathbf{u})$ is an eigenvalue/eigenvector pair for $\delta\mathbf{e}\mathbf{e}^\top-\boldsymbol{\Gamma}$ if and only if 
$(\lambda\mathbf{I}+\boldsymbol{\Gamma})\mathbf{u}=\delta\mathbf{e}\mathbf{e}^\top\mathbf{u}$, and, so long as $\lambda$ is \emph{not} also an eigenvalue of $-\boldsymbol{\Gamma}$, this happens if and only if 
$\lambda$ is a root of the secular equation  $s(\lambda)=\frac1\delta-\mathbf{e}^\top(\lambda\mathbf{I}+\boldsymbol{\Gamma})^{-1}\mathbf{e}=0$. Assertion (c) can be reinterpretted as stating that all solutions to $s(\lambda)=0$ must be positive for all sufficiently large $\delta>0$.
$s(x)<0$ for all $x<-\varsigma_1<0$. 
Weyl interlacing  \cite[Corollary 4.3.9, p. 241]{HornJohnson2013MatrixAnalysis} guarantees that the smallest $\lambda$ for which $s(\lambda)=0$ must occur for $-\varsigma_1<\lambda\leq -\varsigma_2$.   Invertibility of $\boldsymbol{\Gamma}$ guarantees that  $0<-\varsigma_2$ (this is evidently a necessary condition for (c) to hold). The smallest $\lambda$ for which $s(\lambda)=0$ will be positive provided that $\mathbf{e}^\top\boldsymbol{\Gamma}^{-1}\mathbf{e}>0$ in which case, we may take $\delta>1/\mathbf{e}^\top\boldsymbol{\Gamma}^{-1}\mathbf{e}$. (See Figure \ref{fig:fig2}).

To show $\mathbf{e}^\top\boldsymbol{\Gamma}^{-1}\mathbf{e}>0$, suppose to the contrary that $\mathbf{e}^\top\boldsymbol{\Gamma}^{-1}\mathbf{e}\leq 0$.
We have shown above that the left-hand side of
\begin{equation}
\label{eqn:M}
\boldsymbol{\gamma}(\mathbf{t})\,\mathbf{e}^{\!\top}+\mathbf{e}\,\boldsymbol{\gamma}(\mathbf{t})^{\!\top}-\boldsymbol{\Gamma}=\boldsymbol{\gamma}(\mathbf{t})(\mathbf{e}^{\!\top}\boldsymbol{\Gamma}^{-1}\mathbf{e})\boldsymbol{\gamma}(\mathbf{t})^{\!\top}-\left(\boldsymbol{\Gamma}-\boldsymbol{\gamma}(\mathbf{t})\mathbf{e}^{\!\top}\right)\boldsymbol{\Gamma}^{-1}\left(\boldsymbol{\Gamma}-\mathbf{e}\boldsymbol{\gamma}(\mathbf{t})^{\!\top}\right)
\end{equation}
is positive-semidefinite; the right-hand side reorganizes the same expression 
\begin{wrapfigure}[14]{r}{0.5\textwidth}
  \centering 
\vspace{-6mm}
\includegraphics[width=0.48\textwidth] {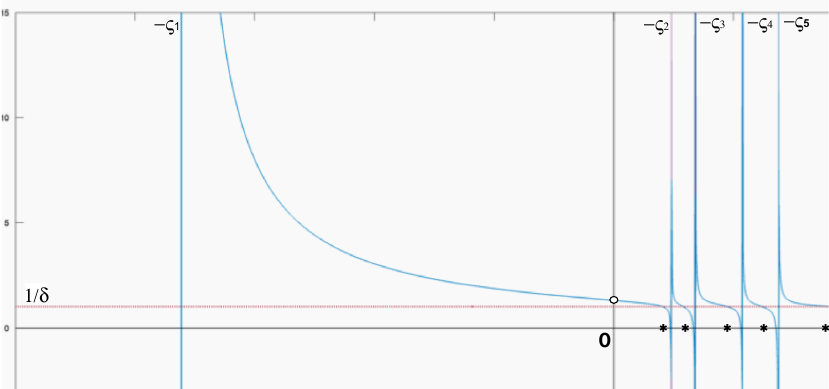}   \caption{\label{fig:fig2}
   Illustration of the argument for Theorem \ref{ThmCondNegDef} (c). The function $\mathbf{e}^\top(s\mathbf{I}+\boldsymbol{\Gamma})^{-1}\mathbf{e}$ (blue curve) has singularities at $-\varsigma_1,-\varsigma_2,\ldots$ and  intersection points with the line $y=\frac1\delta$ (dotted) determine eigenvalues of $\delta\mathbf{e}\mathbf{e}^\top-\boldsymbol{\Gamma}$ (denoted $\star$).  So long as the value of $\mathbf{e}^\top\boldsymbol{\Gamma}^{-1}\mathbf{e}>0$ ($\circ$), there will be a sufficiently large value of $\delta$ such that $\delta\mathbf{e}\mathbf{e}^\top-\boldsymbol{\Gamma}$ is positive-definite.}
\end{wrapfigure}
   and so must also be positive semidefinite. If $\boldsymbol{\Gamma}-\mathbf{e}\boldsymbol{\gamma}(\mathbf{t})^{\!\top}$ is nonsingular,  then $\left(\boldsymbol{\Gamma}-\boldsymbol{\gamma}(\mathbf{t})\mathbf{e}^{\!\top}\right)\boldsymbol{\Gamma}^{-1}\left(\boldsymbol{\Gamma}-\mathbf{e}\boldsymbol{\gamma}(\mathbf{t})^{\!\top}\right)$ has the same number of positive and negative eigenvalues as $\boldsymbol{\Gamma}$, hence it is indefinite and $\mathbf{e}^{\!\top}\boldsymbol{\Gamma}^{-1}\mathbf{e}\leq 0$ implies that the number of negative eigenvalues cannot decrease when $\boldsymbol{\gamma}(\mathbf{t})(\mathbf{e}^{\!\top}\boldsymbol{\Gamma}^{-1}\mathbf{e})\boldsymbol{\gamma}(\mathbf{t})^{\!\top}$ is added. Thus, in this case we must have $\mathbf{e}^{\!\top}\boldsymbol{\Gamma}^{-1}\mathbf{e}>0$.  If $\boldsymbol{\Gamma}-\mathbf{e}\boldsymbol{\gamma}(\mathbf{t})^{\!\top}$ is singular then take $\mathbf{v}\neq\mathbf{0}$ so that $(\boldsymbol{\Gamma}-\mathbf{e}\boldsymbol{\gamma}(\mathbf{t})^{\!\top})\mathbf{v}=\mathbf{0}$. Note that $\boldsymbol{\gamma}(\mathbf{t})^{\!\top}\mathbf{v}\neq 0$ since otherwise $\boldsymbol{\Gamma}\mathbf{v}=\mathbf{0}$ would give a contradiction to the assumed nonsingularity of $\boldsymbol{\Gamma}$. But then $\mathbf{v}^{\!\top}(\boldsymbol{\gamma}(\mathbf{t})(\mathbf{e}^{\!\top}\boldsymbol{\Gamma}^{-1}\mathbf{e})\boldsymbol{\gamma}(\mathbf{t})^{\!\top})\mathbf{v}=(\mathbf{e}^{\!\top}\boldsymbol{\Gamma}^{-1}\mathbf{e})|\boldsymbol{\gamma}(\mathbf{t})^{\!\top}\mathbf{v}|^2\geq 0$ and so we find also in this case, $\mathbf{e}^{\!\top}\boldsymbol{\Gamma}^{-1}\mathbf{e}>0$.  $\Box$

\subsection{Factoring IRF covariance matrices}\label{FactorIRF}
While the Cholesky factorization of the IRF covariance, $\mathbf{M}=\boldsymbol{\gamma}(\mathbf{t})\,\mathbf{e}^{\!\top}+\mathbf{e}\,\boldsymbol{\gamma}(\mathbf{t})^{\!\top}-\boldsymbol{\Gamma}$ (Eqn. \ref{eqn:M}), can be computed directly, one  may need to do this repeatedly for different values of $\mathbf{t}\in\Omega$ and so a preferred approach would avoid repeating redundant aspects of the factorization and ideally would avoid explicit construction of $\mathbf{M}$ while still taking advantage of its structure.  One way of doing this takes into account Theorem \ref{ThmCondNegDef}(c). Choose $\delta>0$ sufficiently large so that $\delta\mathbf{e}\mathbf{e}^\top-\boldsymbol{\Gamma}$ is positive definite and compute the Cholesky factorization, $\delta\mathbf{e}\mathbf{e}^\top-\boldsymbol{\Gamma}=\mathbf{L}_0\mathbf{L}_0^\top$ with $\mathbf{L}_0\in\mathbb{R}^{n\times n}$ lower triangular. 
Given $\mathbf{L}_0$,  one may proceed to compute cheaply multiple ``twisted" factorizations of $\delta\mathbf{e}\mathbf{e}^\top-\begin{bmatrix}
    0 & \boldsymbol{\gamma}(\mathbf{t})^{\!\top} \\
    \boldsymbol{\gamma}(\mathbf{t}) & \boldsymbol{\Gamma}
\end{bmatrix}$ for different choices of $\mathbf{t}$ by observing: \begin{align*}
\delta\mathbf{e}\mathbf{e}^\top-\begin{bmatrix}
    0 & \boldsymbol{\gamma}(\mathbf{t})^{\!\top} \\
    \boldsymbol{\gamma}(\mathbf{t}) & \boldsymbol{\Gamma}
\end{bmatrix}&=\begin{bmatrix}
    \delta & (\delta\mathbf{e}-\boldsymbol{\gamma}(\mathbf{t}))^{\!\top} \\
    \delta\mathbf{e}-\boldsymbol{\gamma}(\mathbf{t}) & \delta\mathbf{e}\mathbf{e}^\top-\boldsymbol{\Gamma}
\end{bmatrix}=\begin{bmatrix}
    \rho(\mathbf{t}) & \mathbf{r}(\mathbf{t})^\top \\
    \mathbf{0} & \mathbf{L}_0
\end{bmatrix}\begin{bmatrix}
    \rho(\mathbf{t}) & \mathbf{0} \\
   \mathbf{r}(\mathbf{t})  & \mathbf{L}_0^\top
\end{bmatrix},
\end{align*}
where $\mathbf{r}(\mathbf{t})$ solves $\mathbf{L}_0\mathbf{r}(\mathbf{t})=\delta\mathbf{e}-\boldsymbol{\gamma}(\mathbf{t})$ and $\rho(\mathbf{t})=\sqrt{\delta-\|\mathbf{r}(\mathbf{t})\|^{2}}$.
$\mathbf{L}_0$ is independent of $\mathbf{t}$ and need only be computed once. A ``twisted"  factorization may be generated for each new value of $\mathbf{t}$ with  added marginal complexity of  only $\mathcal{O}(n^{2})$ floating point operations (as opposed to $\mathcal{O}(n^{3})$ if the decomposition was computed from scratch); only $\rho(\mathbf{t})$ and $\mathbf{r}(\mathbf{t})$ need be retained. 
Recalling that $\mathbf{D}\mathbf{e}=\mathbf{0}$ with $\mathbf{D}=[-\mathbf{e}\quad \mathbf{I}]$ we have
\begin{align*}
\boldsymbol{\gamma}(\mathbf{t})\,\mathbf{e}^{\!\top}+\mathbf{e}&\,\boldsymbol{\gamma}(\mathbf{t})^{\!\top}-\boldsymbol{\Gamma}=
-\mathbf{D}\begin{bmatrix}
    0 & \boldsymbol{\gamma}(\mathbf{t})^{\!\top} \\
    \boldsymbol{\gamma}(\mathbf{t}) & \boldsymbol{\Gamma}
\end{bmatrix}\mathbf{D}^\top\\
=\mathbf{D}&\left(\delta\mathbf{e}\mathbf{e}^\top-\begin{bmatrix}
    0 & \boldsymbol{\gamma}(\mathbf{t})^{\!\top} \\
    \boldsymbol{\gamma}(\mathbf{t}) & \boldsymbol{\Gamma}
\end{bmatrix}\right)\mathbf{D}^\top
=\left(\mathbf{D}\begin{bmatrix}
\rho(\mathbf{t}) & \mathbf{r}(\mathbf{t})^\top \\
    \mathbf{0} & \mathbf{L}_0
\end{bmatrix}\right)
\left(\mathbf{D}\begin{bmatrix}
\rho(\mathbf{t}) & \mathbf{r}(\mathbf{t})^\top \\
    \mathbf{0} & \mathbf{L}_0
\end{bmatrix}\right)^\top,
\end{align*}
and we may complete a Cholesky factorization for 
$\boldsymbol{\gamma}(\mathbf{t})\,\mathbf{e}^{\!\top}+\mathbf{e}\,\boldsymbol{\gamma}(\mathbf{t})^{\!\top}-\boldsymbol{\Gamma}$ by computing a \textsf{QR} decomposition of  
$\left(\mathbf{D}\begin{bmatrix}
    \rho(\mathbf{t}) & \mathbf{r}(\mathbf{t})^\top \\
    \mathbf{0} & \mathbf{L}_0
\end{bmatrix}\right)^\top=\begin{bmatrix}
    -\rho(\mathbf{t})\mathbf{e}^\top  \\
     \mathbf{L}_0^\top- \mathbf{r}(\mathbf{t})\mathbf{e}^\top
\end{bmatrix}$.  This proceeds in two phases: first, we reduce $\mathbf{L}_0^\top- \mathbf{r}(\mathbf{t})\mathbf{e}^\top$, taking advantage of the fact that it is a rank-one modification of an upper triangular matrix:  $(n-1)$
Givens rotations are applied to annihilate  consecutive trailing entries of $\mathbf{r}(\mathbf{t})$, transforming $\mathbf{r}(\mathbf{t})\mapsto \|\mathbf{r}(\mathbf{t})\|\mathbf{e}_1$ while simultaneously mapping $\mathbf{L}_0^\top$ to an upper Hessenberg matrix. In the second phase, two additional passes are made with Givens rotations to annihilate the subdiagonals producing ultimately an upper triangular Cholesky factor. None of the $3n$ Givens rotations generated across both phases need to be retained. The computational complexity of the whole process is approximately $8n^2$ floating point operations to evaluate a Cholesky factor for each $\mathbf{t}$-location considered, as well as $n^3/3$ floating point operations needed for the initial (one time) factorization of  $\delta\mathbf{e}\mathbf{e}^\top-\boldsymbol{\Gamma}=\mathbf{L}_0\mathbf{L}_0^\top$.

Note that the initial factorization of  $\delta\mathbf{e}\mathbf{e}^\top-\boldsymbol{\Gamma}$ will generally be computed at the same time as a sufficiently large value for  $\delta$ is chosen to ensure positive definiteness.  There is some benefit in choosing $\delta$ sufficiently large so that fewer subsequent adjustments need be made in the course of factorization.   Notice that $\delta\mathbf{e}\mathbf{e}^\top-\boldsymbol{\Gamma}\leq \delta n\mathbf{I}-\boldsymbol{\Gamma}$ (with respect to the L\"{o}wner ordering for symmetric matrices), so it is necessary that $\delta n > \varsigma_1$, the largest positive eigenvalue of $\boldsymbol{\Gamma}$ (though this will not necessarily be sufficient).  The magnitude of $\varsigma_1$ is unknown but an upper bound can be easily obtained from $\frac{\mathbf{e}^\top\boldsymbol{\Gamma}\mathbf{e}}{\|\boldsymbol{\Gamma}\mathbf{e}\|^{2}}\geq \varsigma_1$ \cite{beattie1998harmonic}.  This might suggest taking $\delta=\frac1n \frac{\|\boldsymbol{\Gamma}\mathbf{e}\|^{2}}{\mathbf{e}^\top\boldsymbol{\Gamma}\mathbf{e}}$, and indeed this quantity will generally be a very good approximation to $\varsigma_1/n$. Nonetheless, we find  $\delta= \frac{\|\boldsymbol{\Gamma}\mathbf{e}\|^{2}}{\mathbf{e}^\top\boldsymbol{\Gamma}\mathbf{e}}$  often to be a better choice.  A straightforward modification of the Cholesky factorization as found in \cite{chengHigham1998modChol} will produce  additional adjustments to $\delta$ that might be required to guarantee positive definiteness.

\section{GP regression with intrinsic variogram priors}
\label{sec:GPreg}
Having established some background elements, we now discuss  GP regression with intrinsic priors in this section.   Given observations $\{y_1,y_2,\ldots,y_n\}$ of the process  $Z$ at fixed locations $\{\mathbf{s}_1,\mathbf{s}_2,\ldots,\mathbf{s}_n\}\subset\Omega\subset\mathbb{R}^d$, our goal will be to estimate with uncertainty $Z(\mathbf{t})$ at any given $\mathbf{t}\in\Omega$.  We take a Bayesian approach and assume the following hierarchical model,
{\small
\begin{align}
\mbox{\textbf{Data model} }[Y|Z]:&\quad 
    \mathbf{y}=\begin{pmatrix}
        y_1\\y_2\\ \vdots \\y_n
      \end{pmatrix}=\begin{pmatrix}
        Z(\mathbf{s}_1)\\Z(\mathbf{s}_2)\\ \vdots \\Z(\mathbf{s}_n)
\end{pmatrix}+\sigma \mathbf{F}\mathbf{u} 
\quad \mbox{where } \mathbf{u}\sim N(0,\mathbf{I}); \label{dataModel} \\ 
  &\qquad \sigma^{2}\mathbf{F}\mathbf{F}^\top \mbox{is the observation error covariance (factored)} \nonumber \\[2mm]
        \mbox{\textbf{Prior Process model} }[Z]:&\quad Z \mbox{ is a zero-order intrinsic Gaussian process (\textbf{0-IGP})} \label{processModel}\\ 
       &\quad \mbox{ with normally distributed increments and a given}\nonumber\\
       &\qquad \mbox{ variogram:  }\gamma(\mathbf{s},\mathbf{t})=\frac12\mathbb{E}[(Z(\mathbf{s})-Z(\mathbf{t}))^{2}].\nonumber
\end{align} }

\noindent
This can be rewritten compactly as
\begin{equation}
    \mbox{\textbf{Data/Process Model}: }\left\{\begin{array}{ll}
\mathbf{y}=\mathbf{X}\,\mathbb{Z}+\mathbf{F}\mathbf{u},& \mathbf{u}\sim N(0,\mathbf{I}) \\[1mm]
\mathbf{D}\,\mathbb{Z}=\mathbf{G}(\mathbf{t})\mathbf{v},& \mathbf{v}\sim N(0,\mathbf{I})
\end{array}\right.\label{eq:data-process}
\end{equation}
where $\mathbf{X}=\left[\begin{array}{ccccc} 0&1 & 0 &\cdots& 0\\ 0&0 &1 &  & 0 \\
\vdots & & & \ddots & \\ 0&0 & 0 &\cdots& 1\end{array}\right]$ 
and $\mathbf{G}(\mathbf{t})\mathbf{G}(\mathbf{t})^\top$ is a factorization of the IGP covariance, $\boldsymbol{\gamma}(\mathbf{t})\,\mathbf{e}^{\!\top}+\mathbf{e}\,\boldsymbol{\gamma}(\mathbf{t})^{\!\top}-\boldsymbol{\Gamma}$, associated with the increments $\mathbf{D}\,\mathbb{Z}=\begin{pmatrix}
        Z(\mathbf{s}_1)-Z(\mathbf{t})\\
        Z(\mathbf{s}_2)-Z(\mathbf{t})\\ \vdots \\
        Z(\mathbf{s}_n)-Z(\mathbf{t})
 \end{pmatrix}$.

With this hierarchical model in mind, we start by defining estimates for Kriging coefficients with  IGP priors in \sect \ref{sec:IGPKrigCoeff} and detail characteristics of these estimates in \sect \ref{sec:funcKrig} and \sect \ref{sec:flexIncr}.  We conclude in \sect \ref{sec:postReal} by applying  GP regression with IGP priors to assess a realistic, albeit simulated, dataset that emulates key features of \textsf{SWOT} altimetry data\footnote{ \textsf{SWOT} $=$ 
``Surface Water and Ocean Topography" is a joint NASA and CNES Earth-observing satellite mission launched in December, 2022
 designed to survey over 90\% of Earth's surface water \cite{vaze2018surface,SWOTDocs}} .  We highlight computational guidelines for both estimating and simulating IGPs a posteriori.

\subsection{IGP Kriging coefficients}
\label{sec:IGPKrigCoeff}
\noindent
The \emph{intrinsic Gaussian process (IGP) regression estimator} for $Z(\mathbf{t})$ is written in terms of observations as $\widehat{Z}(\mathbf{t})=\sum_{k=1}^n \lambda_k(\mathbf{t})\,y_k$ where $\{\lambda_k(\mathbf{t})\}_1^n$ are the \emph{IGP Kriging coefficients}.  Since $Z(\mathbf{t})=\mathbf{e}_1^T\mathbb{Z}$, the best linear unbiased estimator for $Z(\mathbf{t})$ may be directly computed: 
$$
\widehat{Z}(\mathbf{t})=\frac1{\sigma^{2}}\mathbf{e}_1^T\left(\frac1{\sigma^{2}}\mathbf{X}^T(\mathbf{F}\mathbf{F}^T)^{-1}\mathbf{X}+\mathbf{D}^T(\mathbf{G}(\mathbf{t})\mathbf{G}(\mathbf{t})^T)^{-1}\mathbf{D}\right)^{-1}\mathbf{X}^T(\mathbf{F}\mathbf{F}^T)^{-1}\mathbf{y}.
$$
Noting that $\mathbf{y}=\sum_{k=1}^n y_k\,\mathbf{e}_k$, the Kriging functions are defined by 
\begin{equation}\label{KrigingFunc}
\lambda_k(\mathbf{t})=\frac1{\sigma^{2}}\mathbf{e}_1^T\left(\frac1{\sigma^{2}}\mathbf{X}^T(\mathbf{F}\mathbf{F}^T)^{-1}\mathbf{X}+\mathbf{D}^T(\mathbf{G}(\mathbf{t})\mathbf{G}(\mathbf{t})^T)^{-1}\mathbf{D}\right)^{-1}\mathbf{X}^T(\mathbf{F}\mathbf{F}^T)^{-1}\mathbf{e}_k.
\end{equation}

\begin{theorem}
\label{Thm_AffineKriging_0IRF}
  For the data/process model given above using an IGP variogram prior, the IGP regression estimator for  $Z(\mathbf{t})$ is an affine combination of observation values. That is, $\displaystyle \widehat{Z}(\mathbf{t})=\sum_{k=1}^n \lambda_k(\mathbf{t})y_k$ and  $\displaystyle \sum_{k=1}^n \lambda_k(\mathbf{t})=1$.
\end{theorem}
\noindent
\emph{Proof:} Introduce  $\mathbf{W}(\mathbf{t})=\frac1{\sigma^2}\mathbf{X}^T(\mathbf{F}\mathbf{F}^T)^{-1}\mathbf{X}+\mathbf{D}^T(\mathbf{G}(\mathbf{t})\mathbf{G}(\mathbf{t})^T)^{-1}\mathbf{D}$ so that we may write compactly, $\lambda_k(\mathbf{t})
=\frac1{\sigma^2}\mathbf{e}_1^T\mathbf{W}(\mathbf{t})^{-1} \mathbf{X}^T(\mathbf{F}\mathbf{F}^T)^{-1}\mathbf{e}_k$.  
Observe that $\mathbf{X}\begin{bmatrix}
    \frac1n\\ \mathbf{e}_k
\end{bmatrix}=\mathbf{e}_k$ and $\mathbf{D}\begin{bmatrix}
    1\\ \mathbf{e}
\end{bmatrix}=\mathbf{0}$.  By adding and subtracting $\mathbf{W}(\mathbf{t})^{-1}\mathbf{D}^T(\mathbf{G}(\mathbf{t})\mathbf{G}(\mathbf{t})^T)^{-1}\mathbf{D}$ we find
\begin{align*}
    \frac1{\sigma^2}\mathbf{W}(\mathbf{t})^{-1}\mathbf{X}^T&(\mathbf{F}\mathbf{F}^T)^{-1}\mathbf{X}=\\
    \mathbf{W}(\mathbf{t})^{-1}&\underbrace{(\frac1{\sigma^2}\mathbf{X}^T(\mathbf{F}\mathbf{F}^T)^{-1}\mathbf{X}+\mathbf{D}^T(\mathbf{G}(\mathbf{t})\mathbf{G}(\mathbf{t})^T)^{-1}\mathbf{D})}_{\mathbf{W}(\mathbf{t})}
    -\mathbf{W}(\mathbf{t})^{-1}\mathbf{D}^T(\mathbf{G}(\mathbf{t})\mathbf{G}(\mathbf{t})^T)^{-1}\mathbf{D} \\
    &\qquad\qquad =\mathbf{I}-\mathbf{W}(\mathbf{t})^{-1}\mathbf{D}^T(\mathbf{G}(\mathbf{t})\mathbf{G}(\mathbf{t})^T)^{-1}\mathbf{D}.
\end{align*}
Hence,
\begin{align*}
\displaystyle \sum_{k=1}^n \lambda_k(\mathbf{t})=&\frac1{\sigma^2}\sum_{k=1}^n\mathbf{e}_1^T\mathbf{W}(\mathbf{t})^{-1} \mathbf{X}^T(\mathbf{F}\mathbf{F}^T)^{-1}\mathbf{e}_k=\frac1{\sigma^2}\sum_{k=1}^n\mathbf{e}_1^T\mathbf{W}(\mathbf{t})^{-1}\mathbf{X}^T(\mathbf{F}\mathbf{F}^T)^{-1}\mathbf{X}\begin{bmatrix}
    \frac1n\\ \mathbf{e}_k
\end{bmatrix} \\  
& =\mathbf{e}_1^T\left(\mathbf{I}-\mathbf{W}(\mathbf{t})^{-1}\mathbf{D}^T(\mathbf{G}(\mathbf{t})\mathbf{G}(\mathbf{t})^T)^{-1}\mathbf{D}\right)\sum_{k=1}^n\begin{bmatrix}
    \frac1n\\ \mathbf{e}_k
\end{bmatrix}\\[2mm]
&\quad =\mathbf{e}_1^T\left(\mathbf{I}-\mathbf{W}(\mathbf{t})^{-1}\mathbf{D}^T(\mathbf{G}(\mathbf{t})\mathbf{G}(\mathbf{t})^T)^{-1}\mathbf{D}\right)
\begin{bmatrix}
    1\\ \mathbf{e}
\end{bmatrix}=\mathbf{e}_1^T\begin{bmatrix}
    1\\ \mathbf{e}
\end{bmatrix}=1
\quad\Box
\end{align*}
\subsection{The functional form of Kriging coefficients}
\label{sec:funcKrig}
We study the structure of the Kriging coefficients, $\{\lambda_k(\mathbf{t})\}_{k=1}^n$, more closely, deriving them in terms of the distribution of our basic  increments: 
\begin{equation}\label{KrigingFuncStar}
\lambda_k(\mathbf{t})=\frac1{\sigma^2}\mathbf{e}_1^T\left(\frac1{\sigma^2}\mathbf{X}^T(\mathbf{F}\mathbf{F}^T)^{-1}\mathbf{X}+\mathbf{D}^T(\mathbf{G}(\mathbf{t})\mathbf{G}(\mathbf{t})^T)^{-1}\mathbf{D}\right)^{-1}\mathbf{X}^T(\mathbf{F}\mathbf{F}^T)^{-1}\mathbf{e}_k.
\end{equation}
Partition as before $\mathbf{D}=[-\mathbf{e}\,|\,\mathbf{I}]$ and define  the posterior precision: 
\begin{equation} \label{postPrec}
    \mathbf{W}(\mathbf{t})=\frac1{\sigma^2}\mathbf{X}^T(\mathbf{F}\mathbf{F}^T)^{-1}\mathbf{X}+\mathbf{D}^T(\mathbf{G}(\mathbf{t})\mathbf{G}(\mathbf{t})^T)^{-1}\mathbf{D} =\begin{bmatrix}
    a(\mathbf{t}) & \mathbf{b}(\mathbf{t})^{\!\top} \\
    \mathbf{b}(\mathbf{t}) & \mathbf{C}(\sigma,\mathbf{t})
\end{bmatrix}
\end{equation}
where $a(\mathbf{t})=\mathbf{e}^\top(\mathbf{G}(\mathbf{t})\mathbf{G}(\mathbf{t})^T)^{-1}\mathbf{e}$, $\mathbf{b}(\mathbf{t})=-(\mathbf{G}(\mathbf{t})\mathbf{G}(\mathbf{t})^T)^{-1}\mathbf{e}$, and $\mathbf{C}(\sigma,\mathbf{t})=\frac1{\sigma^2}(\mathbf{F}\mathbf{F}^\top)^{-1}+(\mathbf{G}(\mathbf{t})\mathbf{G}(\mathbf{t})^T)^{-1}$.
With these substitutions, the IGP Kriging coefficient, $\lambda_k(\mathbf{t})$, can be expressed as 
\begin{align*}
\lambda_k(\mathbf{t})=&\frac1{\sigma^2}\mathbf{e}_1^T\begin{bmatrix}
a(\mathbf{t}) & \mathbf{b}(\mathbf{t})^{\!\top} \\
    \mathbf{b}(\mathbf{t}) & \mathbf{C}(\sigma,\mathbf{t})
\end{bmatrix}^{-1}\begin{pmatrix}
   \mathbf{0} \\ (\mathbf{F}\mathbf{F}^T)^{-1}\mathbf{e}_k 
\end{pmatrix}=
\frac1{\sigma^2} \frac{-\mathbf{b}(\mathbf{t})^{\!\top}\mathbf{C}(\sigma,\mathbf{t})^{-1}(\mathbf{F}\mathbf{F}^T)^{-1}\mathbf{e}_k}{a(\mathbf{t})-\mathbf{b}(\mathbf{t})^{\!\top}\mathbf{C}(\sigma,\mathbf{t})^{-1}\mathbf{b}(\mathbf{t})}\\[2mm]
&=\frac{\mathbf{e}^\top(\mathbf{G}\mathbf{G}^T)^{-1}\left(\frac1{\sigma^2}(\mathbf{F}\mathbf{F}^\top)^{-1}+(\mathbf{G}\mathbf{G}^T)^{-1}\right)^{-1}\frac1{\sigma^2}(\mathbf{F}\mathbf{F}^T)^{-1}\mathbf{e}_k}
{\mathbf{e}^\top\left((\mathbf{G}\mathbf{G}^T)^{-1}-(\mathbf{G}\mathbf{G}^T)^{-1}\left(\frac1{\sigma^2}(\mathbf{F}\mathbf{F}^\top)^{-1}+(\mathbf{G}\mathbf{G}^T)^{-1}\right)^{-1}(\mathbf{G}\mathbf{G}^T)^{-1}
\right)\mathbf{e}}
\end{align*}
which finally gives
\begin{equation}
    \label{KrigCoeff:RatForm}
\lambda_k(\mathbf{t})=\frac{\mathbf{e}^\top\left(\sigma^2\,\mathbf{F}\mathbf{F}^\top+\mathbf{G}(\mathbf{t})\mathbf{G}(\mathbf{t})^\top\right)^{-1}\mathbf{e}_k}
{\mathbf{e}^\top\left(\sigma^2\,\mathbf{F}\mathbf{F}^\top+\mathbf{G}(\mathbf{t})\mathbf{G}(\mathbf{t})^\top\right)^{-1}\mathbf{e}}
\end{equation}
This is similar in form to Joseph's rational Kriging predictor and in the same sense that Joseph uses the term, we might expect the IGP Kriging coefficients to have a rational form with respect to the core location parameters $\{\boldsymbol{\gamma}(\mathbf{t},\mathbf{s}_1),\boldsymbol{\gamma}(\mathbf{t},\mathbf{s}_2),\ldots,\boldsymbol{\gamma}(\mathbf{t},\mathbf{s}_n)\}$.
Indeed, one may note that the entries of $\mathbf{G}(\mathbf{t})\mathbf{G}(\mathbf{t})^\top=\boldsymbol{\gamma}(\mathbf{t})\,\mathbf{e}^{\!\top}+\mathbf{e}\,\boldsymbol{\gamma}(\mathbf{t})^{\!\top}-\boldsymbol{\Gamma}$ are \emph{linear} expressions with  respect to the location parameters  $\{\boldsymbol{\gamma}(\mathbf{t},\mathbf{s}_1),\boldsymbol{\gamma}(\mathbf{t},\mathbf{s}_2),\ldots,\boldsymbol{\gamma}(\mathbf{t},\mathbf{s}_n)\}$. Hence the construction of  numerators and denominators in \eqref{KrigCoeff:RatForm} should be expected to each yield ratios of \emph{polynomials} in the location parameters, and so the Kriging coefficients themselves may be expected to be \emph{rational functions} of  the core location parameters, $\{\boldsymbol{\gamma}(\mathbf{t},\mathbf{s}_1),\boldsymbol{\gamma}(\mathbf{t},\mathbf{s}_2),\ldots,\boldsymbol{\gamma}(\mathbf{t},\mathbf{s}_n)\}$.
\subsubsection{Limit-Kriging vs IGP Kriging}
As an illustration, we are able to recover Joseph's limit-Kriging approach as a special case of \eqref{KrigCoeff:RatForm}. From \eqref{limKrigCoeff}, recall that limit-Kriging within the variogram framework produces coefficients defined as 
$
\displaystyle
\lambda_k^{\mathsf{lim}}(\mathbf{t})
=\frac{(\vartheta^{2}\mathbf{e}-\boldsymbol{\gamma}(\mathbf{t}))^\top(\vartheta^{2}\mathbf{e}\mathbf{e}^\top-\boldsymbol{\Gamma})^{-1}\mathbf{e}_k}{(\vartheta^{2}\mathbf{e}-\boldsymbol{\gamma}(\mathbf{t}))^\top(\vartheta^{2}\mathbf{e}\mathbf{e}^\top-\boldsymbol{\Gamma})^{-1}\mathbf{e}}.
$
This will correspond to \eqref{KrigCoeff:RatForm}, if we can find a nonsingular matrix, $\mathbf{G}_{\mathsf{lim}}(\mathbf{t})$ such that 
$$
\left(\sigma^2\mathbf{F}\mathbf{F}^\top+\mathbf{G}_{\mathsf{lim}}(\mathbf{t})\mathbf{G}_{\mathsf{lim}}(\mathbf{t})^{\top}\right) (\vartheta^{2}\mathbf{e}\mathbf{e}^\top-\boldsymbol{\Gamma})^{-1}(\vartheta^{2}\mathbf{e}-\boldsymbol{\gamma}(\mathbf{t}))=\mathbf{e}.
$$
For simplicity consider the noise-free case, where $\sigma\rightarrow 0$ and define scalar functions of $\mathbf{t}$ as 
$$
\begin{array}{c}
f(\mathbf{t})=\mathbf{e}^\top(\vartheta^{2}\mathbf{e}\mathbf{e}^\top-\boldsymbol{\Gamma})^{-1}(\vartheta^{2}\mathbf{e}-\boldsymbol{\gamma}(\mathbf{t})),\\[2mm]
g(\mathbf{t})=(\vartheta^{2}\mathbf{e}-\boldsymbol{\gamma}(\mathbf{t}))^\top(\vartheta^{2}\mathbf{e}\mathbf{e}^\top-\boldsymbol{\Gamma})^{-1}(\vartheta^{2}\mathbf{e}-\boldsymbol{\gamma}(\mathbf{t})),\ \mbox{and}\\[2mm] \displaystyle h(\mathbf{t})=\frac{f(\mathbf{t})+g(\mathbf{t})}{f(\mathbf{t})^2}.
\end{array}
$$
Then define 
{\small
\begin{equation}
    \label{DFP}
\mathbf{G}_{\mathsf{lim}}(\mathbf{t})\mathbf{G}_{\mathsf{lim}}(\mathbf{t})^{\top}=\vartheta^{2}\mathbf{e}\mathbf{e}^\top-\boldsymbol{\Gamma}\ +\ h(\mathbf{t})\,\mathbf{e}\mathbf{e}^\top-\frac1{f(\mathbf{t})}(\vartheta^{2}\mathbf{e}-\boldsymbol{\gamma}(\mathbf{t}))\mathbf{e}^\top-\frac1{f(\mathbf{t})}\mathbf{e}(\vartheta^{2}\mathbf{e}-\boldsymbol{\gamma}(\mathbf{t}))^\top.
\end{equation}
}
With this definition, we can verify directly that 
$$
\mathbf{G}_{\mathsf{lim}}(\mathbf{t})\mathbf{G}_{\mathsf{lim}}(\mathbf{t})^{\top} (\vartheta^{2}\mathbf{e}\mathbf{e}^\top-\boldsymbol{\Gamma})^{-1}(\vartheta^{2}\mathbf{e}-\boldsymbol{\gamma}(\mathbf{t}))=\mathbf{e}.
$$
and so, 
$$
\mathbf{e}^T\left(\mathbf{G}_{\mathsf{lim}}(\mathbf{t})\mathbf{G}_{\mathsf{lim}}(\mathbf{t})^{\top} \right)^{-1}=(\vartheta^{2}\mathbf{e}-\boldsymbol{\gamma}(\mathbf{t}))^\top(\vartheta^{2}\mathbf{e}\mathbf{e}^\top-\boldsymbol{\Gamma})^{-1}.
$$
With this definition for $\mathbf{G}_{\mathsf{lim}}$, the Kriging coefficients in \eqref{KrigCoeff:RatForm} 
will match the limit-Kriging coefficients given in  \eqref{limKrigCoeff}.  The expression given in \eqref{DFP} for $\mathbf{G}_{\mathsf{lim}}\mathbf{G}_{\mathsf{lim}}^\top$ matches the \textsf{DFP} quasi-Newton secant update and is one of the Broyden-class of update formulas (e.g., see \cite{nocedal2006numerical} \S 6.3). A key property of the \textsf{DFP} update is that so long as $f(\mathbf{t})>0$, the right-hand side of \eqref{DFP} is positive-definite. Conversely, if $f(\mathbf{t})\leq 0$, there will be \emph{no} choice of $\mathbf{G}(\mathbf{t})$ in 
\eqref{KrigCoeff:RatForm} that could yield $\lambda_k^{\mathsf{lim}}(\mathbf{t})$.
Taking into account how the covariance of increments relates to the variogram, we have an alternative expression derived earlier:
\begin{align*}
    \mathbf{G}(\mathbf{t})\mathbf{G}(\mathbf{t})^T&=-\mathbf{D}\begin{bmatrix}
    0 & \boldsymbol{\gamma}(\mathbf{t})^{\!\top} \\
    \boldsymbol{\gamma}(\mathbf{t}) & \boldsymbol{\Gamma}
\end{bmatrix}\mathbf{D}^\top=\boldsymbol{\gamma}(\mathbf{t})\,\mathbf{e}^{\!\top}+\mathbf{e}\,\boldsymbol{\gamma}(\mathbf{t})^{\!\top}-\boldsymbol{\Gamma}\\
&= \vartheta^{2}\mathbf{e}\mathbf{e}^\top -\boldsymbol{\Gamma}  + \vartheta^{2}\mathbf{e}\mathbf{e}^\top-(\vartheta^{2}\mathbf{e}-\boldsymbol{\gamma}(\mathbf{t}))\,\mathbf{e}^{\!\top}-\mathbf{e}\,(\vartheta^{2}\mathbf{e}-\boldsymbol{\gamma}(\mathbf{t}))^\top 
\end{align*}  
Note that both $\mathbf{G}(\mathbf{t})\mathbf{G}(\mathbf{t})^T$ and $\mathbf{G}_{\mathsf{lim}}(\mathbf{t})\mathbf{G}_{\mathsf{lim}}(\mathbf{t})^{\top}$ are linear combinations of four $n\times n$ matrices:$\{\vartheta^{2}\mathbf{e}\mathbf{e}^\top -\boldsymbol{\Gamma},   \mathbf{e}\mathbf{e}^\top, (\vartheta^{2}\mathbf{e}-\boldsymbol{\gamma}(\mathbf{t}))\,\mathbf{e}^{\!\top},\mathbf{e}\,(\vartheta^{2}\mathbf{e}-\boldsymbol{\gamma}(\mathbf{t}))^\top\}$, which constitutes a linearly independent set so long as $\boldsymbol{\gamma}(\mathbf{t})$ is \emph{not} a multiple of $\mathbf{e}$. Thus,  
$\mathbf{G}(\mathbf{t})\mathbf{G}(\mathbf{t})^T=\mathbf{G}_{\mathsf{lim}}(\mathbf{t})\mathbf{G}_{\mathsf{lim}}(\mathbf{t})^{\top}$ if and only if their coefficients match: $f(\mathbf{t})=1$ and $h(\mathbf{t})=\vartheta^{2}$, which evidently is an untenable hypothesis. This leads us to conclude that $\mathbf{G}(\mathbf{t})\mathbf{G}(\mathbf{t})^T$ and $\mathbf{G}_{\mathsf{lim}}(\mathbf{t})\mathbf{G}_{\mathsf{lim}}(\mathbf{t})^{\top}$ are manifestly different matrices describing, therefore, manifestly different prior distributions of $\mathbf{D}\,\mathbb{Z}$. 
From this, we conclude that regardless of how well limit-Kriging may perform as a predictor of $Z(\mathbf{t})$, it must conform with a distribution that is inconsistent with the original prior distribution determined by the variogram.  

\subsubsection{The linear form of IGP Kriging coefficients}
We go on to consider the form of \eqref{KrigCoeff:RatForm} in more detail. Note that $\sigma^2\,\mathbf{F}\mathbf{F}^\top+\mathbf{G}(\mathbf{t})\mathbf{G}(\mathbf{t})^T$ is a rank-two change from $\sigma^2\,\mathbf{F}\mathbf{F}^\top-\boldsymbol{\Gamma}$, so by the Sherman-Morrison-Woodbury formula, we can write $\left(\sigma^2\,\mathbf{F}\mathbf{F}^\top+\mathbf{G}(\mathbf{t})\mathbf{G}(\mathbf{t})^T\right)^{-1}$ as a rank-two change to $\left(\sigma^2\,\mathbf{F}\mathbf{F}^\top-\boldsymbol{\Gamma}\right)^{-1}$. It may happen that $\sigma^2\,\mathbf{F}\mathbf{F}^\top-\boldsymbol{\Gamma}$ is singular (recall that $\boldsymbol{\Gamma}$ is  indefinite), however we may introduce $\delta\,\mathbf{e}\mathbf{e}^\top$ and write:
{\small
$$
\sigma^2\,\mathbf{F}\mathbf{F}^\top+\mathbf{G}(\mathbf{t})\mathbf{G}(\mathbf{t})^T=-\left(\frac{\delta}2\mathbf{e}-\boldsymbol{\gamma}\right)\,\mathbf{e}^{\!\top}-\mathbf{e}\,\left(\frac{\delta}2\mathbf{e}-\boldsymbol{\gamma}\right)^{\!\top}+\sigma^2\,\mathbf{F}\mathbf{F}^\top+\delta\,\mathbf{e}\mathbf{e}^\top-\boldsymbol{\Gamma}.
$$
}
We have established in 
Theorem \ref{ThmCondNegDef}, that $\delta>0$ may be chosen large enough to guarantee that $\delta\,\mathbf{e}\mathbf{e}^\top-\boldsymbol{\Gamma}$ is positive definite (and so, also $\sigma^2\,\mathbf{F}\mathbf{F}^\top+\delta\,\mathbf{e}\mathbf{e}^\top-\boldsymbol{\Gamma}$). 
For compactness in the expressions that follow, we will suppress the $\mathbf{t}$ dependence in $\boldsymbol{\gamma}$, assume that $\sigma^2\,\mathbf{F}\mathbf{F}^\top-\boldsymbol{\Gamma}$ is positive definite, and represent $\left(\sigma^2\,\mathbf{F}\mathbf{F}^\top-\boldsymbol{\Gamma}\right)^{-1}=\mathbf{N}$ :
{\small 
\begin{align*}
(\sigma^2\,\mathbf{F}\mathbf{F}^\top+\mathbf{G}(\mathbf{t})\mathbf{G}(\mathbf{t})^T)^{-1}&=\left(\sigma^2\,\mathbf{F}\mathbf{F}^\top-\boldsymbol{\Gamma}+\boldsymbol{\gamma}\,\mathbf{e}^{\!\top}+\mathbf{e}\,\boldsymbol{\gamma}^{\!\top}\right)^{-1}\\
=\mathbf{N}
-\mathbf{N}&[\boldsymbol{\gamma},\,\mathbf{e}]\begin{bmatrix}
    \boldsymbol{\gamma}^\top\mathbf{N}\boldsymbol{\gamma} & 1+ \boldsymbol{\gamma}^\top\mathbf{N}\mathbf{e} \\
1+\mathbf{e}^\top\mathbf{N}\boldsymbol{\gamma} & \mathbf{e}^\top\mathbf{N}\mathbf{e}
\end{bmatrix}^{-1}\begin{bmatrix}
   \boldsymbol{\gamma}^\top\\
    \mathbf{e}^\top
\end{bmatrix}\mathbf{N}
\end{align*} }
This allows us to  compute directly (after some effort)
\begin{align*}
\mathbf{e}^\top(\sigma^2\,\mathbf{F}\mathbf{F}^\top&+\mathbf{G}(\mathbf{t})\mathbf{G}(\mathbf{t})^T)^{-1}\mathbf{e}=
\\
&\mathbf{e}^\top\mathbf{N}\mathbf{e}+\frac{[\mathbf{e}^\top\mathbf{N}\boldsymbol{\gamma}, \mathbf{e}^\top\mathbf{N}\mathbf{e}] 
\begin{bmatrix}
-\mathbf{e}^\top\mathbf{N}\mathbf{e}
     & 1+ \boldsymbol{\gamma}^\top\mathbf{N}\mathbf{e} \\
   1+\mathbf{e}^\top\mathbf{N}\boldsymbol{\gamma} & -\boldsymbol{\gamma}^\top\mathbf{N}\boldsymbol{\gamma}
\end{bmatrix}
\begin{pmatrix}
\boldsymbol{\gamma}^\top\mathbf{N}\mathbf{e}\\ \mathbf{e}^\top\mathbf{N}\mathbf{e}\end{pmatrix}}
{\boldsymbol{\gamma}^\top\mathbf{N}\boldsymbol{\gamma}\,\mathbf{e}^\top\mathbf{N}\mathbf{e}-(\boldsymbol{\gamma}^\top\mathbf{N}\mathbf{e}\,+\,1)^{2}}\\[2mm] 
&\qquad = \frac{-\mathbf{e}^\top\mathbf{N}\mathbf{e}}{\boldsymbol{\gamma}^\top\mathbf{N}\boldsymbol{\gamma}\,\mathbf{e}^\top\mathbf{N}\mathbf{e}-(\boldsymbol{\gamma}^\top\mathbf{N}\mathbf{e}\,+\,1)^{2}}
\end{align*}
and 
\begin{align*}
\mathbf{e}^\top(\sigma^2\,\mathbf{F}\mathbf{F}^\top&+\mathbf{G}(\mathbf{t})\mathbf{G}(\mathbf{t})^T)^{-1}\mathbf{e}_k=
\\
&\mathbf{e}^\top\mathbf{N}\mathbf{e}_k+\frac{[\mathbf{e}^\top\mathbf{N}\boldsymbol{\gamma}, \mathbf{e}^\top\mathbf{N}\mathbf{e}] 
\begin{bmatrix}
-\mathbf{e}^\top\mathbf{N}\mathbf{e}
     & 1+ \boldsymbol{\gamma}^\top\mathbf{N}\mathbf{e} \\
   1+\mathbf{e}^\top\mathbf{N}\boldsymbol{\gamma} & -\boldsymbol{\gamma}^\top\mathbf{N}\boldsymbol{\gamma}
\end{bmatrix}
\begin{pmatrix}
\boldsymbol{\gamma}^\top\mathbf{N}\mathbf{e}_k\\ \mathbf{e}^\top\mathbf{N}\mathbf{e}_k\end{pmatrix}}
{\boldsymbol{\gamma}^\top\mathbf{N}\boldsymbol{\gamma}\,\mathbf{e}^\top\mathbf{N}\mathbf{e}-(\boldsymbol{\gamma}^\top\mathbf{N}\mathbf{e}\,+\,1)^{2}}\\[2mm] 
&\qquad = \frac{\mathbf{e}^\top\mathbf{N}\mathbf{e}\boldsymbol{\gamma}^\top\mathbf{N}\mathbf{e}_k-(1+\boldsymbol{\gamma}^\top\mathbf{N}\mathbf{e})\mathbf{e}^\top\mathbf{N}\mathbf{e}_k}{\boldsymbol{\gamma}^\top\mathbf{N}\boldsymbol{\gamma}\,\mathbf{e}^\top\mathbf{N}\mathbf{e}-(\boldsymbol{\gamma}^\top\mathbf{N}\mathbf{e}\,+\,1)^{2}}.
\end{align*}
Taking ratios in \eqref{KrigCoeff:RatForm}, we find an explicit expression for the Kriging coefficients:
\begin{equation}\label{KrigCoeffLinearExpr}
\lambda_k(\mathbf{t})=(1+\boldsymbol{\gamma}(\mathbf{t})^\top\mathbf{N}\mathbf{e})\frac{\mathbf{e}^\top\mathbf{N}\mathbf{e}_k}{\mathbf{e}^\top\mathbf{N}\mathbf{e}}-\boldsymbol{\gamma}(\mathbf{t})^\top\mathbf{N}\mathbf{e}_k. 
\end{equation}
Two comments are worth making about \eqref{KrigCoeffLinearExpr}: First, despite the formal \emph{rational} form anticipated in \eqref{KrigCoeff:RatForm}, the expression in \eqref{KrigCoeffLinearExpr} is \emph{linear} with respect to the location parameters $\{\boldsymbol{\gamma}(\mathbf{t},\mathbf{s}_1),\boldsymbol{\gamma}(\mathbf{t},\mathbf{s}_2),\ldots,\boldsymbol{\gamma}(\mathbf{t},\mathbf{s}_n)\}$; and secondly, this expression was derived with a simplifying assumption that amounted to taking $\delta=0$, but in fact, there will be some computational benefits in taking $\delta>0$, as we discuss below. Nonetheless, the final expression in \eqref{KrigCoeffLinearExpr} for the Kriging coefficients is \emph{independent} of $\delta$. 

We define the ``noise-free Kriging coefficients" as the limiting values of the coefficients, $\lambda_k(\mathbf{t})$, 
as  observation noise vanishes, $\sigma\rightarrow 0$:
$\displaystyle \lambda_k^0(\mathbf{t})=\lim_{\sigma\rightarrow 0}\lambda_k(\mathbf{t})$.
Notice that as 
 $\sigma\rightarrow 0$, $\mathbf{N}=\left(\sigma^2\,\mathbf{F}\mathbf{F}^\top-\boldsymbol{\Gamma}\right)^{-1}\rightarrow -\boldsymbol{\Gamma}^{-1}$, so from \eqref{KrigCoeffLinearExpr}, 
\begin{equation} \label{noiseFreeKrigCoef}
\lambda_k^0(\mathbf{t}) =
\boldsymbol{\gamma}(\mathbf{t})^\top\boldsymbol{\Gamma}^{-1}\mathbf{e}_k+
(1-\mathbf{e}^\top\boldsymbol{\Gamma}^{-1}\boldsymbol{\gamma}(\mathbf{t}))\frac{\mathbf{e}^\top\boldsymbol{\Gamma}^{-1}\mathbf{e}_k}{\mathbf{e}^\top\boldsymbol{\Gamma}^{-1}\mathbf{e}}.
\end{equation}
This implies that noise-free Kriging is interpolation: 
\begin{theorem}\label{noiseFreeKrigIntrp}
    In the limiting case of $\sigma\rightarrow 0$, $\lambda_k^0(\mathbf{s}_{\ell})=\left\{\begin{array}{ll}
    1 &\mbox{ for }\ell=k\\
    0 & \mbox{ otherwise}
    \end{array}\right. .$\\ Hence $\displaystyle \widehat{Z}(\mathbf{t})=\sum_{k=1}^n \lambda_k^0(\mathbf{t})\,y_k$  for exact observations, and  $\widehat{Z}(\mathbf{s}_{\ell})=y_{\ell},$ for $\ell=1,2,\ldots,n$
\end{theorem}
\begin{proof}\!:
    Note that as $\mathbf{t}\rightarrow\mathbf{s}_{\ell}$, $\boldsymbol{\gamma}(\mathbf{t})\rightarrow \boldsymbol{\Gamma}\mathbf{e}_{\ell}$.  Then from \eqref{noiseFreeKrigCoef}, 
\begin{align*} \lambda_k^0(\mathbf{s}_{\ell})&=\lim_{\mathbf{t}\rightarrow\mathbf{s}_{\ell}}\lambda_k^0(\mathbf{t})=
(\boldsymbol{\Gamma}\mathbf{e}_{\ell})^\top\boldsymbol{\Gamma}^{-1}\mathbf{e}_k+
(1-\mathbf{e}^\top\boldsymbol{\Gamma}^{-1}(\boldsymbol{\Gamma}\mathbf{e}_{\ell}))\frac{\mathbf{e}^\top\boldsymbol{\Gamma}^{-1}\mathbf{e}_k}{\mathbf{e}^\top\boldsymbol{\Gamma}^{-1}\mathbf{e}}\\
&=\mathbf{e}_{\ell}^\top\mathbf{e}_k+
(1-\mathbf{e}^\top\mathbf{e}_{\ell}))\frac{\mathbf{e}^\top\boldsymbol{\Gamma}^{-1}\mathbf{e}_k}{\mathbf{e}^\top\boldsymbol{\Gamma}^{-1}\mathbf{e}}=\mathbf{e}_{\ell}^\top\mathbf{e}_k
\end{align*}
from which the remaining assertions follow. 
\end{proof}
\subsection{Flexibility in choosing increments}
\label{sec:flexIncr}
We have implemented the process model above using increments taken relative to the target quantity $Z(\mathbf{t})$: 
$$
\mathbf{D}\,\begin{pmatrix}
        Z(\mathbf{t})\\ Z(\mathbf{s}_1)\\Z(\mathbf{s}_2)\\ \vdots \\Z(\mathbf{s}_n)
\end{pmatrix}=\begin{pmatrix}
        Z(\mathbf{s}_1)-Z(\mathbf{t})\\
        Z(\mathbf{s}_2)-Z(\mathbf{t})\\ \vdots \\
        Z(\mathbf{s}_n)-Z(\mathbf{t})
 \end{pmatrix}
 \quad\mbox{with}\quad \mathbf{D}=\left[\begin{array}{ccccc} -1&1 & 0 &\cdots& 0\\ -1&0 &1 &  & 0 \\
\vdots & & & \ddots & \\ -1&0 & 0 &\cdots& 1\end{array}\right]
$$
However, we might prefer to  implement a different process model, say,  $\widehat{\mathbf{D}}\,\mathbb{Z}=\widehat{\mathbf{G}}(\mathbf{t})\mathbf{u}$ with $\mathbf{u}\sim N(\mathbf{0},\mathbf{I})$.  Will this change our Kriging coefficients and hence, our predictors ? 

\begin{theorem}
    Suppose $\widehat{\mathbf{D}}\in\mathbb{R}^{n\times(n+1)}$ satisfies  $\mathsf{rank}(\widehat{\mathbf{D}})=n$ and $\widehat{\mathbf{D}}\mathbf{e}=\mathbf{0}$. Let $\widehat{\mathbf{G}}(\mathbf{t})\in\mathbb{R}^{n\times n}$ be an invertible matrix for each $\mathbf{t}\in\Omega$. Partition $\widehat{\mathbf{D}}$ as $\widehat{\mathbf{D}}=[\hat{\mathbf{d}}\ \mathbf{J}]$.  Then  $\mathbf{J}$ is invertible,  and the process model determined by $\widehat{\mathbf{D}}\,\mathbb{Z}=\widehat{\mathbf{G}}(\mathbf{t})\mathbf{u}$ with $\mathbf{u}\sim N(\mathbf{0},\mathbf{I})$  is equivalent to the process model specified by \eqref{processModel} 
    if and only if
   $\widehat{\mathbf{G}}(\mathbf{t})\widehat{\mathbf{G}}(\mathbf{t})^\top=\mathbf{J}\left(\boldsymbol{\gamma}(\mathbf{t})\,\mathbf{e}^{\!\top}+\mathbf{e}\,\boldsymbol{\gamma}(\mathbf{t})^{\!\top}-\boldsymbol{\Gamma}\right)\mathbf{J}^\top$, where we refer to quantities defined in \eqref{varDef}. Under these conditions, the prior precision of $\mathbb{Z}$ induced by this process model remains unchanged from that of \eqref{dataModel}-\eqref{processModel}. In particular, $\widehat{\mathbf{D}}^\top\left(\widehat{\mathbf{G}}(\mathbf{t})\widehat{\mathbf{G}}(\mathbf{t})^\top\right)^{-1}\widehat{\mathbf{D}}=\mathbf{D}^\top\left(\mathbf{G}(\mathbf{t})\mathbf{G}(\mathbf{t})^\top\right)^{-1}\mathbf{D}$ and the Kriging coefficients are as in  \eqref{KrigingFunc}. 
\end{theorem}
\begin{proof}
 Note first that the conditions, $\mathsf{rank}(\widehat{\mathbf{D}})=n$ and $\widehat{\mathbf{D}}\mathbf{e}=\mathbf{0}$, imply that \emph{every} set of $n$ columns of $\widehat{\mathbf{D}}$ must be linearly independent and the $n \times n$ submatrix, $\mathbf{J}$, of $\widehat{\mathbf{D}}$ consisting of the last $n$ columns must be invertible.  Let $\mathbf{f}=\mathbf{J}^{-1}\hat{\mathbf{d}}$, so that $\mathbf{J}^{-1}\widehat{\mathbf{D}}=[\mathbf{f}\,|\,\mathbf{I}]$
Since $\widehat{\mathbf{D}}\mathbf{e}=\mathbf{0}$, then  $\mathbf{f}=-\mathbf{e}$ and so $\mathbf{J}\mathbf{D}=\widehat{\mathbf{D}}$. If $\widehat{\mathbf{G}}(\mathbf{t})\widehat{\mathbf{G}}(\mathbf{t})^\top=\mathbf{J}\left(\boldsymbol{\gamma}(\mathbf{t})\,\mathbf{e}^{\!\top}+\mathbf{e}\,\boldsymbol{\gamma}(\mathbf{t})^{\!\top}-\boldsymbol{\Gamma}\right)\mathbf{J}^\top=\mathbf{J}\mathbf{G}(\mathbf{t})\mathbf{G}(\mathbf{t})^\top\mathbf{J}^\top$
then $\widehat{\mathbf{D}}\,\mathbb{Z}=\widehat{\mathbf{G}}(\mathbf{t})\mathbf{u}$ will be equivalent to $\mathbf{D}\,\mathbb{Z}=\mathbf{G}(\mathbf{t})\mathbf{u}$, and conversely. Note that it need \emph{not} be true that $\widehat{\mathbf{G}}=\mathbf{J}\mathbf{G}$.
\end{proof}
To illustrate this, suppose we wish to use consecutive differences: 
$$
\widehat{\mathbf{D}}\,\begin{pmatrix}
        Z(\mathbf{t})\\ Z(\mathbf{s}_1)\\Z(\mathbf{s}_2)\\ \vdots \\Z(\mathbf{s}_n)
\end{pmatrix}=\begin{pmatrix}
        Z(\mathbf{s}_1)-Z(\mathbf{t})\\
        Z(\mathbf{s}_2)-Z(\mathbf{s}_1)\\
         Z(\mathbf{s}_3)-Z(\mathbf{s}_2)\\
         \vdots \\
        Z(\mathbf{s}_n)-Z(\mathbf{s}_{n-1})
 \end{pmatrix}
 \quad\mbox{with}\quad \widehat{\mathbf{D}}=\left[\begin{array}{cccccc} 
 -1&1 & 0 & 0&\cdots& 0\\
 0& -1& 1 &  0&  & 0 \\
 0& 0&  -1& 1 &  & 0 \\
\vdots & & & \ddots & \\ 0 &0 & 0 &\cdots& -1 &1\end{array}\right],
$$
then we could follow similar steps as above in order to formulate the covariance for these increments directly in terms of variogram data.  Start with an algebraic identity: 
{\small 
\begin{align*}
(Z(\mathbf{s}_{i+1})-Z(\mathbf{s}_i))(Z(\mathbf{s}_{j+1})-Z(\mathbf{s}_j)) =\frac12&\left[
(Z(\mathbf{s}_i)-Z(\mathbf{s}_{j+1}))^{2}
+(Z(\mathbf{s}_{i+1})-Z(\mathbf{s}_j))^{2}\right.\\
& \left.-(Z(\mathbf{s}_{i+1})-Z(\mathbf{s}_{j+1}))^{2}
-(Z(\mathbf{s}_i)-Z(\mathbf{s}_j))^{2}\right]
\end{align*} }
Then taking expectations: 
{\small 
\begin{align*}
\mathbb{E}[(Z(\mathbf{s}_{i+1})-Z(\mathbf{s}_i))(Z(\mathbf{s}_{j+1})-Z(\mathbf{s}_j))] =&\gamma(\mathbf{s}_i,\mathbf{s}_{j+1})
+\gamma(\mathbf{s}_{i+1},\mathbf{s}_j) -\gamma(\mathbf{s}_{i+1},\mathbf{s}_{j+1})
-\gamma(\mathbf{s}_i,\mathbf{s}_j)\\
&= - \begin{bmatrix}
    -1 & 1
\end{bmatrix} \begin{bmatrix}
    \gamma(\mathbf{s}_i,\mathbf{s}_j) & \gamma(\mathbf{s}_i,\mathbf{s}_{j+1})\\
   \gamma(\mathbf{s}_{i+1},\mathbf{s}_j) & \gamma(\mathbf{s}_{i+1},\mathbf{s}_{j+1})
\end{bmatrix}\begin{bmatrix}
    -1 \\ 1
\end{bmatrix}.
\end{align*} }
Expanding this out into matrix form:
$$
\mathsf{cov}(\widehat{\mathbf{D}}\,\mathbb{Z})=
-\widehat{\mathbf{D}}\begin{bmatrix}
    0 & \boldsymbol{\gamma}(\mathbf{t})^{\!\top} \\
    \boldsymbol{\gamma}(\mathbf{t}) & \boldsymbol{\Gamma}
\end{bmatrix}\widehat{\mathbf{D}}^\top
$$
Not surprizingly, changing which increments are modeled  will change the associated increment covariances (e.g., $\mathsf{cov}(\mathbf{D}\,\mathbb{Z})\neq\mathsf{cov}(\widehat{\mathbf{D}}\,\mathbb{Z})$). However, so long as one chooses nondegenerate increments, the associated increment covariances will be related via nonsingular congruence and, critically, the prior precision of $\mathbb{Z}$ will be \emph{independent} of this choice.  For example,  $\mathbf{D}$ and $\widehat{\mathbf{D}}$ are related by: 
\begin{equation} \label{congruenceMap}
\mathbf{J}\,\mathbf{D}=\widehat{\mathbf{D}}\ \mbox{ with}\quad 
\mathbf{J}=\left[\begin{array}{ccccc} 
 1& 0 &  & \cdots &0 \\
 -1 & 1& 0 &  & 0 \\
 0 & -1& 1 &  & 0 \\
\vdots & &\ddots & \ddots & \vdots \\ 
0 &0  &\cdots&-1 & 1  
\end{array}\right]\in\mathbb{R}^{n\times n},\quad\mbox{and}\quad\mathbf{J}^{-1}=\left[\begin{array}{ccccc} 
 1& 0 &  & \cdots &0 \\
 1 & 1& 0 &  & 0 \\
\vdots & & & \ddots & \vdots \\ 
1 &1  & 1 &\cdots& 1  
\end{array}\right]\in\mathbb{R}^{n\times n}.
\end{equation}

The increment covariances are related as 
$$
\mathbf{J}\,\mathbf{D}\begin{bmatrix}
    0 & \boldsymbol{\gamma}(\mathbf{t})^{\!\top} \\
    \boldsymbol{\gamma}(\mathbf{t}) & \boldsymbol{\Gamma}
\end{bmatrix}\mathbf{D}^\top\,\mathbf{J}^\top=\widehat{\mathbf{D}}\begin{bmatrix}
    0 & \boldsymbol{\gamma}(\mathbf{t})^{\!\top} \\
    \boldsymbol{\gamma}(\mathbf{t}) & \boldsymbol{\Gamma}
\end{bmatrix}\widehat{\mathbf{D}}^\top,
$$
which is a nonsingular congruence. 
The prior precision of the IRF $\mathbb{Z}$ is: 
{\small \begin{align*}
-\mathbf{D}^\top\left(\mathbf{D}\begin{bmatrix}
    0 & \boldsymbol{\gamma}(\mathbf{t})^{\!\top} \\
    \boldsymbol{\gamma}(\mathbf{t}) & \boldsymbol{\Gamma}
\end{bmatrix}\mathbf{D}^\top\right)^{-1}\mathbf{D}
&=-\mathbf{D}^\top\mathbf{J}^\top\left(\mathbf{J}\mathbf{D}\begin{bmatrix}
    0 & \boldsymbol{\gamma}(\mathbf{t})^{\!\top} \\
    \boldsymbol{\gamma}(\mathbf{t}) & \boldsymbol{\Gamma}
\end{bmatrix}\mathbf{D}^\top\mathbf{J}^\top\right)^{-1}\mathbf{J}\mathbf{D}\\
&\qquad =-\widehat{\mathbf{D}}^\top\left(\widehat{\mathbf{D}}\begin{bmatrix}
    0 & \boldsymbol{\gamma}(\mathbf{t})^{\!\top} \\
    \boldsymbol{\gamma}(\mathbf{t}) & \boldsymbol{\Gamma}
\end{bmatrix}\widehat{\mathbf{D}}^\top\right)^{-1}\widehat{\mathbf{D}},
\end{align*}}%
where the first expression is derived using increments taken relative to our regression target $Z(\mathbf{t})$, the second exploits the identity, $\mathbf{J}^{-1}\mathbf{J}=\mathbf{I}$, and the  final expression is derived using consecutive increments.  This means that the \textsf{IGP} Kriging coefficients will be \emph{independent} of  the choice of increments  used to define them.

\bigskip

\subsection{Computing posterior realizations}
\label{sec:postReal} 

IGPs have been used in traditional Kriging analyses where (possibly noisy) observations of an unknown spatial field $Z(s)$ are made on a finite set of locations $\mathbf{y}_i = Z(\mathbf{s}_i)+\epsilon_i$, $i=1,\ldots,n$. 
From these direct observations, the goal is to estimate $Z(\mathbf{t})$, with uncertainty, for $\mathbf{t} \in \Omega$.  Uncertainty might be represented in the form of a predictive (posterior) distribution for $Z(\mathbf{t})$, or realizations from this distribution.  In this section, we do both, but assume IGP priors.  

Moving beyond traditional Kriging analyses, \sect \ref{sec:IGPKrigCoeff}- \sect \ref{sec:flexIncr} offered representations for an IGP  prior that can be incorporated into Bayesian hierarchical models. For example, one could specify a prior density for an IGP ${Z}$ so that on a lattice of  $N$ unobserved locations $\mathbf{t}\in\{\mathbf{t}_1,\ldots,\mathbf{t}_{N}\}$ we predict values,
$\mathbf{Z} = \{Z(\mathbf{t}_1),\ldots,Z(\mathbf{t}_{N})\}$, taking into account
\[
[Z|\lambda_z] \propto \lambda_z^\frac{N}{2} \exp\{ -\half \lambda_z \mathbf{Z}^\top \mathbf{W} \mathbf{Z} \}
\]
where 
$\mathbf{W} =\mathbf{D}^T(\mathbf{D} 
(\delta\mathbf{e}\mathbf{e}^\top-\boldsymbol{\Gamma}) 
\mathbf{D}^T)^{-1}\mathbf{D}$ and $\lambda_z$ scales the prior precision matrix $\mathbf{W}$ (e.g., \cite{besag1995bayesian,haran2011gaussian}).

If we require predictions at a large number of locations $\{\mathbf{t}_1,\ldots,\mathbf{t}_{N}\}$, we can use the conditional normal rules to obtain the distribution for $\mathbf{Z} = (Z(\mathbf{t}_1),\ldots,Z(\mathbf{t}_{N}))^T$ given the observations $\mathbf{y}$ associated with 
locations $\mathbf{s}_1,\ldots,\mathbf{s}_{n}$.
Taking  
$$
\boldsymbol{\tilde{\Gamma}} =
\begin{bmatrix} 
\boldsymbol{\tilde\Gamma}_{\mathbf{tt}} & \boldsymbol{\tilde\Gamma}_\mathbf{ts} \\
\boldsymbol{\tilde\Gamma}_{\mathbf{st}} & \boldsymbol{\tilde\Gamma}_\mathbf{ss}
\end{bmatrix} =
\delta \mathbf{e}\mathbf{e}^\top -\begin{bmatrix} 
\boldsymbol{\Gamma}_{\mathbf{tt}} & 
\boldsymbol{\Gamma}_\mathbf{ts} \\
\boldsymbol{\Gamma}_{\mathbf{st}} & 
\boldsymbol{\Gamma}_\mathbf{ss}
\end{bmatrix} = 
\delta \mathbf{e}\mathbf{e}^\top -\boldsymbol{\Gamma}
$$ 
where $\boldsymbol{\Gamma}$ is the variogram matrix evaluated on the joint vector of observation and prediction locations $(\mathbf{t},\mathbf{s})$. If
$\sigma^2\mathbf{F}\mathbf{F}^\top$ is the covariance of the observation error, the posterior mean and variance are given by
\begin{eqnarray}
\nonumber
\boldsymbol{\mu}_{\rm post} &=& \boldsymbol{\tilde\Gamma}_{\mathbf{ts}} (\sigma^2\mathbf{F}\mathbf{F}^\top + \boldsymbol{\tilde\Gamma}_{\mathbf{ss}})^{-1}\mathbf{y}\\
\nonumber
\boldsymbol{\Sigma}_{\rm post} &=& 
\boldsymbol{\tilde\Gamma}_\mathbf{tt} - \boldsymbol{\tilde\Gamma}_{\mathbf{ts}}
(\sigma^2\mathbf{F}\mathbf{F}^\top + \boldsymbol{\tilde\Gamma}_{\mathbf{ss}})^{-1}
\boldsymbol{\tilde\Gamma}_\mathbf{st}
\end{eqnarray}
The posterior mean $\boldsymbol{\mu}_{\rm post}$ is the Kriging predictor for $\{Z(\mathbf{t}_1),\ldots,Z(\mathbf{t}_{N})\}$; the Bayesian solution gives a posterior distribution for the IGP $Z(t)$ at the prediction locations $\mathbf{t}\in\{\mathbf{t}_1,\ldots,\mathbf{t}_{N}\}$ so that $\mathbf{Z} \sim N(\boldsymbol{\mu}_{\rm post},\boldsymbol{\Sigma}_{\rm post})$. Given $\boldsymbol{\mu}_{\rm post}$ and $\boldsymbol{\Sigma}_{\rm post}$, posterior realizations of 
$\mathbf{Z}$ can be produced (e.g., as in the $5^{\rm th}$ column of Figure \ref{fig:fig1}) using a Cholesky factor of $\boldsymbol{\Sigma}_{\rm post}=\mathbf{R}_{\mathsf{post}}^\top\mathbf{R}_{\mathsf{post}})$, which would then be used to generate realizations as  $\mathsf{realization}= \boldsymbol{\mu}_{\rm post} + \mathbf{R}_{\mathsf{post}}^\top \mathbf{v}$, where $\mathbf{v}$ is a vector of standard normal draws.

The principal effort in computing the posterior mean and realizations will be to compute a Cholesky factor of the posterior covariance $\boldsymbol{\Sigma}_{\rm post}=\mathbf{R}_{\mathsf{post}}^\top\mathbf{R}_{\mathsf{post}}$. We must first compute a factorization of 
$\sigma^2\,\mathbf{F}\mathbf{F}^\top+\delta\,\mathbf{e}\mathbf{e}^\top-\boldsymbol{\Gamma}_{\mathbf{ss}}$ with a suitably chosen $\delta>0$. This has been discussed in \sect \ref{FactorIRF}, and we assume that both a suitable $\delta>0$ and an associated Cholesky factorization of $\delta\,\mathbf{e}\mathbf{e}^\top-\boldsymbol{\Gamma}_{\mathbf{ss}}=\mathbf{L}_\mathbf{s}\mathbf{L}_\mathbf{s}^\top$ has been obtained. 
In order to factor $\sigma^2\,\mathbf{F}\mathbf{F}^\top+\delta\,\mathbf{e}\mathbf{e}^\top-\boldsymbol{\Gamma}$, write $\sigma^2\,\mathbf{F}\mathbf{F}^\top+\delta\,\mathbf{e}\mathbf{e}^\top-\boldsymbol{\Gamma}=\begin{bmatrix}
  \mathbf{L}_\mathbf{s} &  \sigma\mathbf{F} 
\end{bmatrix}\begin{bmatrix}
   \mathbf{L}_\mathbf{s}^\top\\ \sigma\mathbf{F}^\top  
\end{bmatrix}$, and compute the \textsf{QR} factorization of $\begin{bmatrix}
    \mathbf{L}_\mathbf{s}^\top\\
    \sigma\mathbf{F}^\top 
\end{bmatrix}$: $\mathbf{Q}^\top\begin{bmatrix}
    \mathbf{L}_\mathbf{s}^\top\\
    \sigma\mathbf{F}^\top 
\end{bmatrix}=\begin{bmatrix}
    \mathbf{R}_\mathbf{s}\\
    \mathbf{0} 
\end{bmatrix}$ with upper triangular $\mathbf{R}_\mathbf{s}$ which constitutes the Cholesky factor of $\sigma^2\,\mathbf{F}\mathbf{F}^\top+\delta\,\mathbf{e}\mathbf{e}^\top-\boldsymbol{\Gamma}_{\mathbf{ss}}$. The unitary component $\mathbf{Q}$ need not be retained.  Again refering to \sect \ref{FactorIRF}, suppose we have computed a factorization of $\boldsymbol{\tilde\Gamma}_\mathbf{tt}=\delta \mathbf{e}\mathbf{e}^\top -\boldsymbol{\Gamma}_\mathbf{tt}=\mathbf{L}_\mathbf{t}\mathbf{L}_\mathbf{t}^\top$. We wish to produce a Cholesky factor $\mathbf{R}_{\mathsf{post}}$ that satisfies:  $\mathbf{R}_{\mathsf{post}}^\top\mathbf{R}_{\mathsf{post}}=\mathbf{L}_\mathbf{t}\mathbf{L}_\mathbf{t}^\top-\boldsymbol{\tilde\Gamma}_{\mathbf{ts}}(\mathbf{R}_\mathbf{s}^\top\mathbf{R}_\mathbf{s})^{-1}  \boldsymbol{\tilde\Gamma}_{\mathbf{st}}$.  This task is equivalent to ``Cholesky downdating" and amounts to computing a unitary matrix $\widehat{\mathbf{Q}}$ such that $\widehat{\mathbf{Q}}^\top
\begin{bmatrix}
\mathbf{R}_{\mathsf{post}}\\
    \mathbf{R}_\mathbf{s}^{-\top}\boldsymbol{\tilde\Gamma}_{\mathbf{st}}
\end{bmatrix}=\begin{bmatrix}
\mathbf{R}_{\mathsf{s}}\\
    \mathbf{0}
\end{bmatrix}$. As before, the unitary component $\widehat{\mathbf{Q}}$ need not be retained.  We suppress details here but much of this process is standard and one can refer \cite{bojanczyk1987note} and 
\cite[\S 4.3]{stewart1998matrix} for detailed information. 

In Fig.\ref{fig:swot}, we illustrate IGP Kriging with a  second and more involved example than the one motivating the paper in  \sect \ref{sec:sec1}.  For this new example, we model simulated sea surface height (SSH) measurements collected by the satellite, SWOT \cite{SWOTDocs}.  The goal is to use a small set of observations to estimate the SSH throughout a complete region.   For example,  we plot posterior realizations of SSH that are produced given a track of measurements from a simulated  SWOT satellite pass \cite{yaremchuk2023effect,yaremchuk2024block}.
\begin{figure}[th]
  \centerline{
   \includegraphics[width=5.2in,angle=0] {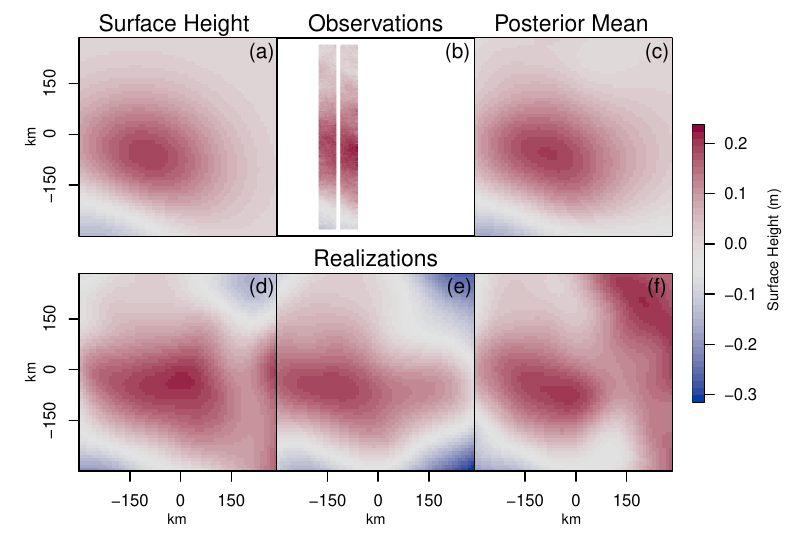}}
   \caption{\label{fig:swot} (a) true synthetic sea surface height (SSH) field; (b) satellite measurements; (c) posterior mean given the observations and variogram; (d)-(f) posterior realizations of SSH over the prediction region.}
\end{figure}
This synthetic example uses an observation error covariance matrix developed for  SWOT satellite measurements, $\sigma^2\mathbf{F}\mathbf{F}^\top$, along with a 2-d IGP prior to produce posterior realizations of SSH over a rectangular section of ocean.  Here we use a variogram resulting from convolving and scaling 2-d Brownian motion
\[
\gamma(\mathbf{d}) = 2r\sigma_z^2  \cdot \Gamma\bigg(\frac{3}{2}\bigg)
 \left[ 
  _1F_1 \left( -\frac{1}{2}; 1; - \frac{||\mathbf{d}||^2}{4r^2}-1 \right)
\right]
\]
where $_1F_1()$ is the confluent hypergeometric function of the first kind, $\sigma^2_z$ is set to 0.00008 m$^2$, and $r$ is set to 50km.
Given the satellite measurements $\mathbf{y}$, their modeled measurement error covariance matrix $\sigma^2\mathbf{FF^\top}$, and the variogram prior described above producing  $\boldsymbol{\Gamma}$, the resulting posterior mean and three realizations are shown in Fig.\ref{fig:swot}.

\section{Summary }
\label{sec:summary}
This work highlights connections between interpolation via stationary GPs, limit kriging, rational kriging, rational interpolation, Shepard interpolation, and ($0^{\rm th}$-order) IGPs.
While some material in \sect 3 and \sect 4 has been discussed in earlier references (e.g., see \cite{matheron1973intrinsic,chiles2012geostatistics}); we present these ideas from a traditional GP perspective so that IGP-based computations can be more readily taken up and implemented by those who model with standard GPs.

The overriding theme of this work  suggests IGPs as an alternative to traditional stationary GPs, as well as to other variants such as limit kriging and rational kriging.  Like other variants we discussed, the resulting posterior mean/Kriging estimates do not show a pull  towards a stationary process mean. The UQ advantages of traditional, stationary GPs are also readily available in prior specifications with IGPs.  Working with IGPs comes with some appealing parsimony as well -- the resulting interpolants (or predictions) are linear both with respect to observations and location parameters; there is no marginal mean parameter to estimate; and the dependence encoded in the intrinsic variogram is often simpler (i.e. fewer parameters) than that of their stationary counterparts.

Section  \ref{sec:postReal} shows how posterior/predictive simulations resulting from IGPs can be carried out directly through judicious use of orthogonal transformations (in order to extract Cholesky factorizations), providing an alternative to turning bands (e.g. \cite{emery2006tbsim}).  While beyond the scope of this work, we note that recent advances that facilitate the use of GPs in large data settings (e.g. Vecchia approximation \cite{vecchia1988estimation,katzfuss2021general}) and nonlinear problems (e.g. elliptical slice sampling \cite{murray2010elliptical}) are readily amenable to posteriors resulting from IGP-based formulations.  

Finally, we note there is not a large collection of $0^{\rm th}$ IGP variograms to choose from in practice; this contrasts with the broad set of choices available for stationary GPs.  If IGPs prove to have advantages over stationary GP modeling, a larger menu of established intrinsic variograms could prove helpful.

\bibliography{RatEst4InGP_biblio}
\end{document}